\documentclass[12pt]{amsart}
\usepackage{amssymb,amscd,amsmath}
\usepackage[makeroom]{cancel}
\usepackage{tikz}
\usepackage{tikz-cd}
\usepackage[all]{xy} 
\usepackage{enumerate}
\usepackage{lmodern,bm}
\usepackage{hyperref}
\hypersetup{
    colorlinks=true,
    linkcolor=teal,
    urlcolor=teal,
    citecolor=teal
}
\usepackage{yfonts}
\newtheorem{fact}{Fact}[section]

\newtheorem{theorem}[fact]{Theorem}

\newtheorem{lemma}[fact]{Lemma}
\newtheorem{corollary}[fact]{Corollary}
\newtheorem{remark}[fact]{Remark}
\newtheorem{definition}[fact]{Definition}
\newtheorem{example}[fact]{Example}

\def\Hom{{\rm Hom}}
\def\End{{\rm End}}
\def\T{{\bf T}}
\def\vt{\vartheta}
\def\vvv{{\rm v}}

\def\vme{{\scriptscriptstyle\vee}}
\def\ee{{\rm e}}
\def\iii{{\rm i}}
\DeclareMathOperator\type{{\tt type}}
\def\id{{\rm id}}

\DeclareMathOperator\Fl{\mathcal F\ell}
\DeclareMathOperator\N{{\mathbb N}}
\def\C{{\mathbb C}}
\def\Q{{\mathbb Q}}
\DeclareMathOperator\Z{{\mathbb Z}}
\DeclareMathOperator\PP{{\mathbb P}}

\DeclareMathOperator{\GL}{GL}
\def\B{{\rm B}}

\def\ulu{\diamond}
\def\Cbar{{\overline{\mathfrak C}}{}}
\def\Cc{{\mathfrak C}}
\def\Ccal{{\mathcal C}}
\def\Ep{\mathcal E\mkern-3.7mu\ell\mkern-2mu\ell}
\def\Elclas{\Ep}
\def\pa{\mathcal P}
\def\Spe{{\mathfrak S}}
\def\TL{{\mathcal T}}
\def\ii{{\rm i}}
\def\FF{F}
\def\g{{x}}
\def\li{{y}}
\def\mub{{\boldsymbol{\mu}}}
\def\xb{{\boldsymbol{x}}}
\def\AA{{\mathfrak A}}
\def\w{{\underline w}}

\title{Link patterns and elliptic Hecke algebra }

\author{Andrzej Weber}
\address{Institute of Mathematics, University of Warsaw, Poland}
\email{aweber@mimuw.edu.pl}
\thanks{The author is supported by Polish National Science Center grant number 2022/47/B/ST1/01896. The author would like to thank the Isaac Newton Institute for Mathematical Sciences, Cambridge, for support and hospitality during the programme {\it  New equivariant methods in algebraic and differential geometry} where work on this paper was partially realized. This work was supported by EPSRC grant no EP/R014604/1}

\begin{document}

\begin{abstract}
We compare three families of geometric objects: Schubert varieties in flag manifolds, matrix Schubert varieties, and Borel orbits of 2-nilpotent matrices. The first family is indexed by permutations, the second by partial permutations, and the third -- the most general -- by link patterns.

Each of the geometric objects mentioned above carries a characteristic class in equivariant elliptic cohomology, defined in the framework provided by Borisov and Libgober. We introduce a Hecke-type algebra that gives inductive formulas for computing the equivariant elliptic classes of link patterns. This requires extending the action of the Hecke algebra to menage partial permutations and link patterns. In studying the extended action, an important role is played by the associated quadratic forms and by the action of reflections on forms. Also, we analyze the specialization of equivariant elliptic classes of link patterns to the corresponding classes of Schubert varieties.
\end{abstract}

\maketitle
\section{Introduction}
Elliptic genera and related characteristic classes have been studied since the  nineteen-eighties  for smooth manifolds, particularly in connection with formal group laws. They have played an important role as tools for constructing homomorphisms from the cobordism ring to the ring of modular forms. A significant extension of this theory to singular algebraic varieties began with the work of Borisov and Libgober \cite{BoLi}.
The elliptic genus can be viewed as a deformation of Hirzebruch’s $\chi_y$-genus. It is defined only for a class of varieties with mild singularities, and the same restriction applies to the associated characteristic classes. Elliptic characteristic classes specialize to motivic Chern classes and, consequently, to Chern-Schwartz-MacPherson classes.
To apply the Borisov-Libgober construction to Schubert varieties in a flag variety, it is necessary to introduce a parameter that deforms the boundary divisor, since Schubert varieties are typically too  singular. This deformation parameter can be identified with a fractional line bundle, that is, an element of the rational Picard group.
\medskip

The idea of applying Hecke-type algebras to the computation of characteristic classes of Schubert varieties originates in \cite{AM, AMSS1, AMSS2} and was later adapted to elliptic classes in \cite{RiWe1}. More recent reformulations and developments have appeared in \cite{LZZ, ZhZh2, LZZ2025}.
Our aim is to describe the algebra governing the computation of elliptic characteristic classes of Schubert varieties and to extend these results to matrix Schubert varieties and Borel orbits of 2-nilpotent matrices. By separating purely algebraic properties from the specific features of the flag variety, we obtain a simple description of the elliptic algebra and clarify the role of its parameters.
The algebra we construct is generated by a version of divided difference operators and can be compared with the algebra of Ginzburg-Kapranov-Vasserot \cite{GKV}. It is generated by quotients of Jacobi theta functions, or more precisely, by a two-parameter elliptic function considered by Zagier \cite{Zag}. The main difference compared to \cite{GKV} is the presence of parameters in the braid relations, analogous to those appearing in the Yang-Baxter equation \cite{Fel}, \cite[\S12.1]{ChariPressley}.

\medskip
Our construction extends the methods worked out for K-theoretic classes in \cite{RuWe} and \cite{KoWe2}.
We will define elements of the equivariant elliptic cohomology of $\Hom(\C^m,\C^m)$ associated to Borel orbits of 2-nilpotent matrices.  
Since $\Hom(\C^m,\C^m)$ is contractible, its elliptic equivariant cohomology coincides with the cohomology of a point. Still we have to be careful which model of  equivariant elliptic cohomology we chose, because there are various possibilities. For our purposes it is convenient to identify the equivariant elliptic cohomology of a point with a subalgebra of meromorphic functions generated by the theta functions. This allows both: to perform concrete computations and view elements of equivartiat elliptic cohomology as sections of sheaves over a product of elliptic curves, as in \cite[\S2]{Ganter}. The parameters $\mu$ and $h$ correspond to additional factors in the product of the elliptic curves, and they have no easy topological explanation. On the other hand the additional parameters are indispensable from the point of view of stable envelopes and related mirror duality, as it is apparent from \cite{AgOk, RTV, RiWe1, RiWe2}.   

\medskip

The set of 2-nilpotent matrices contains an important subset, consisting of upper-triangular block $(n\!\times\!n)$--matrices (when  $m=2n)$. This allows considering matrix Schubert varieties as a special case of Borel orbits of 2-nilpotent matrices. 
Our elliptic classes specialize to the K-theoretic twisted motivic Chern classes of the matrix Schubert varieties defined in \cite{KoWe2}. Moreover, the elliptic classes, after a normalization (essentially the same as  in \cite{RTV} applied to weight functions) become  the elliptic classes of the Schubert varieties, which are the elliptic stable envelopes for maximal torus action, in the sense of \cite{AgOk}.
\medskip

In our calculus we implicitly apply localization theorem for the torus action. Thus, the localized  equivariant elliptic cohomology is the home of our calculus.
An element of the localized equivariant cohomology of $\Hom(\C^m,\C^m)$ is a section of a line bundle over the product of
${m+1}$ copies of an elliptic curve  and depends additionally on so-called  dynamical parameters and a free variable $h$. All together is considered as a section of a line bundle over a bigger product of elliptic curves. We might formally take a direct sum over all relevant vector bundles, as in \cite{LZZ}, but we do not need it, since we only consider pure elements -- sections of individual line bundles.
\medskip

The main goal of this paper is to show that geometric considerations arising from the study of Schubert varieties naturally lead to the construction of a Hecke algebra with parameters. We restrict throughout to the case of the Lie type $A$. 
We compute the Borisov-Libgober elliptic classes of the 2-nilpotent Borel orbits in $\Hom(\C^m,\C^m)$. 

\begin{theorem}[Summary] Let's fix integers $m$ and $r$. 
\begin{itemize}
\item We define in section \S\ref{bundletype} a subalgebra of meromorphic functions in $m+r+2$ variables, denoted by $\AA_{m,r}$. 
\item The algebra $\AA_{m,r}$ is equipped with a { grading} defined by the integral quadratic forms in $m+r+2$ variables. The homogeneous elements are called  {\it pure functions}. 
\item We define operations $\Cc^\mu_i$, $i=1,2,\dots m-1$ depending on the parameters $\mu\in (\C^*)^{r}$. The operations satisfy { braid relations with parameters \eqref{ellbraid} and a} flip relation \eqref{flip}.  
\item For a pure element of  $f\in\AA_{m,r}$ and  $i=1,2,\dots m-1$ there { exists} a unique parameter $\mu$, such that $\Cc^\mu_i(f)$ is pure.
\item The action of $\Cc_i^\mu$ permutes the Borisov-Libgober equivariant elliptic classes of 2-nilpotent upper-triangular Borel orbits in $\Hom(\C^m,\C^m)$.
\item The construction extends to elliptic classes of 2-nilpotent Borel orbits, not necessarily upper-triangular, thanks to the extension of the torus by an extra $\C^*$ factor.
\end{itemize}
\end{theorem}

\medskip
In fact, a version of elliptic Demazure-type operators was introduced in \cite{RiWe1}. There, we discussed the braid relation in \S9, while the {\it flip relation} \eqref{flip} remained implicit. This flip relation has a clear geometric origin, and we provide a geometric proof of it in \S\ref{basicexample}.
The resulting collection of operators acting on meromorphic functions can be viewed as generating an algebra, and one could, in principle, assign a name to this algebra. However, we do not pursue this direction here, as our primary focus is on the action of specific operators on { concrete  {\it pure functions}}.
The resulting action of Demazure operators leads to a unification of the following:
\begin{itemize}
\item elliptic weight functions of \cite{RTV},
\item elliptic classes of Schubert varieties, \cite{RiWe1}
\item twisted motivic Chern class, \cite{KoWe2},
\item CSM and motivic Chern classes of upper-triangular square-zero B-orbits, \cite{RuWe}.
\end{itemize}
\medskip
Embedding $\End(\C^n)$ into $\End(\C^{2n})$ as upper left corner matrices allows to consider matrix Schubert varieties as a special case of 2-nilpotent $\B$-orbits.  We obtain a family of elliptic  functions, which satisfy two recursions related to left and right Demazure-Lusztig operations (or equivalently R-recursion and Bott-Samelson recursion). 
This uniform point of view sheds a light on duality proven in \cite{RiWe2}.
\medskip

We emphasize that the aim of this paper is not to introduce a new version of elliptic algebra. Rather, the main point is the observation that geometric objects, together with the operations on their elliptic cohomology, naturally give rise to an algebraic structure that closely resembles algebras previously appearing in the literature.

\medskip

I would like to thank Jakub Koncki for careful and critical reading of the first version of the paper and providing helpful suggestions.
I would like to thank the reviewer for the valuable and insightful comments that helped improve the clarity and organization of the paper.

\tableofcontents

\section{Background}
One of the basic themes of Schubert calculus is to link homological invariants of Schubert varieties with characteristic classes of tautological bundles. The best known example of such connection is the inductive procedure of computing fundamental classes applying divided differences. The Schubert varieties in the complete flag variety $\Fl_n=\GL_n/\B_n$ are indexed by permutations $w\in\Spe_n$: the Schubert variety $X_w$ is the closure of the $\B_n$-orbit of the permutation matrix $M_w$. Here $\B_n$ denotes the Borel subgroup of $\GL_n$, which we choose to be the group of invertible, upper-triangular matrices. The cohomology ring of $\Fl_n$ is generated by the first Chern classes of the tautological bundles. Denote by $\kappa$  \begin{equation*}\kappa:\Z[\li_1,\li_2,\dots,\li_n]\twoheadrightarrow H^*(\Fl_n)\end{equation*} the corresponding surjection. Bernstein, Gelfand and Gelfand \cite{BGG} have applied the divided difference operators acting on polynomials 
\begin{equation*}\label{partial}(\partial _if)(\li_1,\li_2,\dots,\li_n)=\frac{f(\li_1,\dots,\li_i,\li_{i+1},\dots,\li_n)-f(\li_1,\dots, \li_{i+1},\li_i,\dots,\li_n)}{\li_i-\li_{i+1}}\end{equation*}
to compute inductively the fundamental classes $[X_w]$. If a polynomial $f_w$ represents $[X_w]$, i.e., 
$$\kappa(f)=[X_w]\,,$$
then
\begin{equation}\label{BGG}\kappa(\partial_w(f))=\begin{cases}[X_{ws_i}]&\text{ if }\dim(X_{ws_i})>\dim(X_w)\\
0&\text{ if }\dim(X_{ws_i})<\dim(X_w)\end{cases}\,.\end{equation}
Here $s_i$ is the elementary transposition $(i,i+1)$ for $0<i<n$. The operations $\partial_i$ generate an algebra, called nil-Hecke algebra. The following relations are satisfied
$$\begin{matrix}\partial_i\partial_j=\partial_j\partial_i,& \text{if }|i-j|>1\hfill\,,\\ \\
\partial_i\partial_j\partial_i=\partial_j\partial_i\partial_j,&\text{if }|i-j|=1\hfill&\text{(braid relation)},\\ \\ 
\partial_i^2=0,&&\text{(quadratic relation)}.\end{matrix}$$
The action of the operations $\partial_i$ can serve to compute fundamental classes in the torus--equivariant cohomology $H^*_\T(\Fl_n)$. The equivariant cohomology admits the Borel presentation
\begin{equation}\label{Kirwan}\kappa:\Z[\g_1,\g_2,\dots,\g_n,\li_1,\li_2,\dots \li_n]\twoheadrightarrow H^*_\T(\Fl_n)\simeq \Z[\g_1,\g_2,\dots,\g_n]\otimes_{S_n}\Z[\li_1,\li_2,\dots \li_n]\,,\end{equation}
where $S_n$ is the ring of symmetric polynomials in $n$ variables. The $\g$-variables, called the equivariant variables, are generators of $H^*_\T(pt)$. The divided differences act on both sets of variables. The operations  acting on $\g$-variables are denoted by $\partial_i^\g$ and the operations  acting on $\li$-variables are denoted by $\partial_i^\li$. In equivariant cohomology the formula \eqref{BGG} is satisfied for the operations $\partial_i^\li$ while for the divided difference acting on $\g$-variables we have 
\begin{equation}\label{LBGG}\kappa(\partial^\g(f_w))=\begin{cases}-[X_{s_iw}]&\text{ if }\dim(X_{s_iw})>\dim(X_w)\\
0&\text{ if }\dim(X_{s_iw})<\dim(X_w)\end{cases}\,,\end{equation}
see \cite[Theorem 1.1]{IkMiNa}.
The question arises: is there a geometric interpretation of the operations $\partial^\g_i$ and $\partial^\li_i$ acting on polynomials, not passing to the quotient algebra? It turns out that
with a suitable choice of the starting point the divided difference operations compute the fundamental classes of the matrix Schubert variety, which is the closure of the orbit  $\B_nM_w\B_n\subset \End(\C^n)$, see \cite{FeRi}. The equivariant cohomology
$$H^*_{\T\times\T}(\End(\C^n))\simeq H^*_{\T\times\T}(pt)\simeq \Z[\g_1,\g_2,\dots,\g_n,\li_1,\li_2,\dots \li_n]$$ and the operations $\partial^\li_i$, $\partial^\g_i$ reflect certain geometric operations on matrices.
Further we mix the $\g$-variables with $\li$-variables setting $\g_{i+n}=\li_i$ for $i=1,2,\dots n$. We obtain a nil-Hecke algebra with  generators $\partial^\g_i$ for $i=1,2,\dots,2n-1$. Its geometric meaning was unveiled  in \cite{KnZJ1, KnZJ2} and \cite{RuWe}. We  describe this construction in the next two sections.

\section{Combinatorics governing geometry}
We embed $\End(\C^n)$ into $\End(\C^{2n})$
$$\iota:A\mapsto \begin{pmatrix}0&A\\0&0\end{pmatrix}\,.$$
The $2n\times 2n$ matrix $\iota(A)$ is 2-nilpotent, i.e.~$\iota(A)^2=0$. We observe that the image  of the matrix Schubert cell $\B_nM_w\B_n$ is equal to the $\B_{2n}$ orbit of $\iota(A)$ with $\B_{2n}$ acting by conjugation. Therefore, instead of $\B_n\times \B_n$ orbits in $\End(\C^n)$ we study $\B_{m}$ orbits in the set of 2-nilpotent matrices 
$$\left\{N\in \End(\C^m)\;|\; N^2=0\right\}\,.$$
In general, we do not assume that $m$ is even.
By \cite{BoRe},
 \cite[Th.~7.3.1]{BePe} there are finitely many $\B_m$ orbits. Each orbit contains exactly one matrix of a particular shape: there is at most one   nonzero entry (and it is normalized to 1) at each column and row.   
 We represent the orbits by link patterns. 
 \begin{definition} We fix an integers $r\geq0$ and $m\geq2r$. A  link pattern of rank $r$ is a set of pairs $\{(a_1,b_1),(a_2,b_2),\dots,(a_r,b_r)\}$, such that $a_i, b_i\in\{1,2,\dots,m\}$ are all different numbers. \end{definition}
 We read the link pattern  $\{(7,1),(8,2)\}$  as $7\mapsto 1$, $8\mapsto 2$:
 \vskip13pt

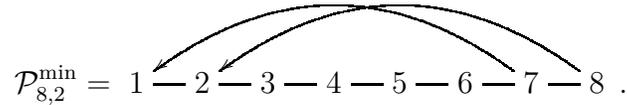
$$
\pa^{\min}_{8,2}=\xymatrix@-1pc{
1\ar@{-}[r]\ar@{<-}@/^2.5pc/[rrrrrr]&
2\ar@{-}[r]\ar@{<-}@/^2.5pc/[rrrrrr]&
3\ar@{-}[r]&
4\ar@{-}[r]&
5\ar@{-}[r]&
6\ar@{-}[r]&
7\ar@{-}[r]&
8}\,.$$
It represents the orbit of the matrix 
$$N_{8,2}^{\min}=\hbox{\arraycolsep=1.7pt\def\arraystretch{0.8}\tiny$\left[\begin{array}{cccccccc}
0&0&0&0&0&0&\bf 1&0\\
0&0&0&0&0&0&0&\bf 1 \\
0&0&0&0&0&0&0&0 \\
0&0&0&0&0&0&0&0\\
0& 0&0&0&0&0&0&0\\        
0&0&0&0&0&0&0&0\\
0&0&0&0&0&0&0&0\\
0&0&0&0&0&0&0&0
\end{array}
\right]$}
$$
This orbit has minimal dimension among $\B_8$ orbits in $\End(\C^8)$ of 2-nilpotent matrices of rank 2.
The link pattern $\{(2,5),(7,4)\}$
\vskip15pt
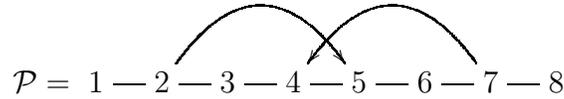
$$
\pa=\xymatrix@-1pc{
1\ar@{-}[r]&
2\ar@{-}[r]\ar@{->}@/^2.5pc/[rrr]&
3\ar@{-}[r]&
4\ar@{-}[r]\ar@{<-}@/^2.5pc/[rrr]&
5\ar@{-}[r]&
6\ar@{-}[r]&
7\ar@{-}[r]&
8}$$
represents the orbit of the  matrix
$$M_w N^{\min}_{8,2}M_w^{-1}=\hbox{\arraycolsep=1.7pt\def\arraystretch{0.8}\tiny$\left[\begin{array}{cccccccc}
0&0&0&0&0&0&0&0\\
0&0&0&0&0&0&0&0 \\
0&0&0&0&0&0&0&0 \\
0&0&0&0&0&0&\bf 1&0\\
0&\bf 1&0&0&0&0&0&0\\        
0&0&0&0&0&0&0&0\\
0&0&0&0&0&0&0&0\\
0&0&0&0&0&0&0&0
\end{array}
\right]$}$$
 for some $w\in \Spe_8\,.$
In what follows we will denote the conjugation action $M_wN\,M_w^{-1}$ by $w\cdot N$.
\medskip

We summarize the connection between geometry and combinatorics with the following table 
$$\begin{matrix}\rm Permutations\\[20pt] 
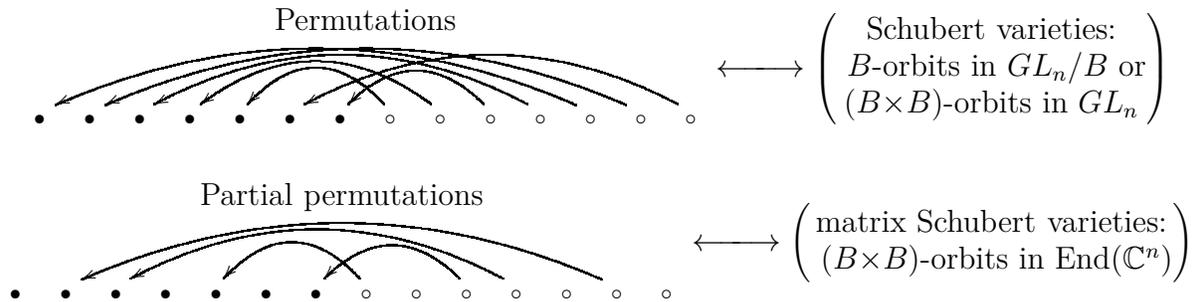
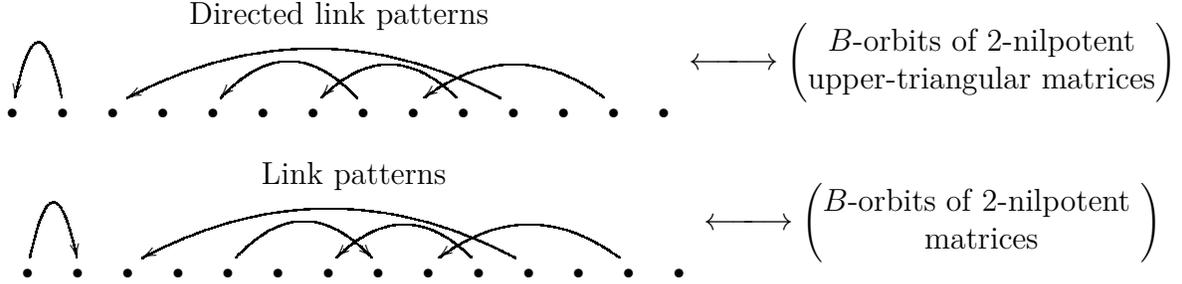
\xymatrix@-1.4pc{
_\bullet\ar@{--}[r]\ar@{<-}@/^2pc/[rrrrrrrrrrr] &
_\bullet\ar@{--}[r]\ar@{<-}@/^2pc/[rrrrrrrrrrr] &
_\bullet\ar@{--}[r]\ar@{<-}@/^1.8pc/[rrrrrrrr] &
_\bullet\ar@{--}[r]\ar@{<-}@/^1.6pc/[rrrrr] &
_\bullet\ar@{--}[r]\ar@{<-}@/^1.4pc/[rrr] &
_\bullet\ar@{--}[r]\ar@{<-}@/^1.8pc/[rrrrrrrr] &_\bullet\ar@{<-}@/^1.3pc/[rrr] &_\circ\ar@{--}[r] &
_\circ\ar@{--}[r] &
_\circ\ar@{--}[r] &
_\circ\ar@{--}[r] &
_\circ\ar@{--}[r] &
_\circ\ar@{--}[r] &
_\circ
}\end{matrix}
\longleftarrow\hskip-5pt\longrightarrow\begin{pmatrix}\text{Schubert varieties:}\\ \text{ $\B_n$-orbits in $\GL_n/\B_n$ or}\\
\text{$(\B_n\smash\times\smash \B_n)$-orbits in $GL_n$}\end{pmatrix}$$

$$\begin{matrix}\rm Partial ~~permutations\\[20pt]
\xymatrix@-1.4pc{
_\bullet\ar@{--}[r] &
_\bullet\ar@{--}[r]\ar@{<-}@/^2pc/[rrrrrrrrrrr] &
_\bullet\ar@{--}[r]\ar@{<-}@/^1.8pc/[rrrrrrrr] &
_\bullet\ar@{--}[r] &
_\bullet\ar@{--}[r]\ar@{<-}@/^1.4pc/[rrr] &
_\bullet\ar@{--}[r] &_\bullet\ar@{<-}@/^1.3pc/[rrr] &
_\circ\ar@{--}[r] &
_\circ\ar@{--}[r] &
_\circ\ar@{--}[r] &
_\circ\ar@{--}[r] &
_\circ\ar@{--}[r] &
_\circ\ar@{--}[r] &
_\circ
}\end{matrix}
\longleftarrow\hskip-5pt\longrightarrow\begin{pmatrix}\text{matrix Schubert varieties:}\\ \text{ $(\B_n\smash\times\smash \B_n)$-orbits in $\End(\C^n)$}\end{pmatrix}$$

$$\begin{matrix}\rm
 Directed~~link~~patterns\\[20pt]
\xymatrix@-1.4pc{
_\bullet\ar@{--}[r]\ar@{<-}@/^2pc/[r] &
_\bullet\ar@{--}[r] &
_\bullet\ar@{--}[r]\ar@{<-}@/^1.8pc/[rrrrrrrr] &
_\bullet\ar@{--}[r] &
_\bullet\ar@{--}[r]\ar@{<-}@/^1.4pc/[rrr] &
_\bullet\ar@{--}[r] &_\bullet\ar@{--}[r]
\ar@{<-}@/^1.3pc/[rrr] &_\bullet\ar@{--}[r] &
_\bullet\ar@{--}[r]\ar@{<-}@/^1.3pc/[rrrr] &
_\bullet\ar@{--}[r] &
_\bullet\ar@{--}[r] &
_\bullet\ar@{--}[r] &
_\bullet\ar@{--}[r] &_\bullet
}
\end{matrix}
\longleftarrow\hskip-5pt\longrightarrow\begin{pmatrix}\text{$\B_m$-orbits of 2-nilpotent}\\ \text{upper-triangular $m\smash\times\smash m$ matrices}\end{pmatrix}$$

$$\begin{matrix}\rm
 Link~~patterns\\[20pt]
 \xymatrix@-1.4pc{
_\bullet\ar@{--}[r]\ar@{->}@/^2pc/[r] &
_\bullet\ar@{--}[r] &
_\bullet\ar@{--}[r]\ar@{<-}@/^1.8pc/[rrrrrrrr] &
_\bullet\ar@{--}[r] &
_\bullet\ar@{--}[r]\ar@{->}@/^1.4pc/[rrr] &
_\bullet\ar@{--}[r] &_\bullet\ar@{--}[r]
\ar@{<-}@/^1.3pc/[rrr] &_\bullet\ar@{--}[r] &
_\bullet\ar@{--}[r]\ar@{<-}@/^1.3pc/[rrrr] &
_\bullet\ar@{--}[r] &
_\bullet\ar@{--}[r] &
_\bullet\ar@{--}[r] &
_\bullet\ar@{--}[r] &_\bullet
}
\end{matrix}
\longleftarrow\hskip-5pt\longrightarrow\begin{pmatrix}\text{$\B_m$-orbits  of 2-nilpotent }\\ \text{$m\smash\times\smash m$ matrices }\end{pmatrix}$$

We need link patterns with labelled arrows hence we consider sequences of pairs (source, target), instead of sets of pairs. 
Formally we consider injective functions from
$\{1,2,\dots,r\}$ to the set  of pairs:
\begin{definition} We fix an integers $r\geq0$ and $m\geq2r$. A labelled link pattern of rank $r$ is a sequence of pairs $\pa=((a_1,b_1),(a_2,b_2),\dots,(a_r,b_r))$, such that $a_i, b_i\in\{1,2,\dots,m\}$ are all different numbers. \end{definition}

Clearly, the product of permutation groups $\Spe_r\times \Spe_m$ acts on the set of all functions
$$\{1,2,\dots,r\}\longrightarrow \{1,2,\dots,m\}^2\,,$$ preserving those, which represent link patterns.
The permutations of the set $\{1,2,\dots,r\}$ will be denoted with the superscript $\mu$.

\section{Characteristic classes of 2-nilpotent orbits}\label{tabela}
The study of homological invariants of $\B$-orbits in 2-nilpotent matrices was initiated by \cite{KnZJ1,KnZJ2}, where homology and K-theory classes were presented as an effect of an action of certain algebras. Further results of \cite{RuWe} apply to more sophisticated invariants: Chern-Schwartz-MacPherson classes and motivic Chern classes of \cite{BSY}.  The table below summarizes the Demazure-type operations in the convention
used in \cite{RuWe}:
\begin{equation*}\begin{matrix}
\text{homology fundamental class}&\beta^{\rm coh}_if=\frac1
{x_{i+1}-{x_i}}f+
\frac{1}{x_{i}-x_{i+1}}s_if\,,&(\beta_i^{\rm coh})^2=0\\[10pt]
\text{K-theory fundamental class}&\beta_if=\frac1
{1-e^{x_{i+1}-x_i}}f+
\frac1{1-e^{x_{i}-x_{i+1}}}s_if\,,&\beta_i^2=\beta_i\\[15pt]
\text{CSM classes}&
\TL^{\rm coh}_i(f)=\frac1
{x_{i+1}-x_i}f+
\frac{1+x_{i}-x_{i+1}}{x_{i}-x_{i+1}} s_if\,,
&(\TL^{\rm coh})^2=id\\[15pt]
\text{motivic Chern classes}&
\TL_i(f)=\frac{(1+y)e^{x_{i+1}-x_i}}
{1-e^{x_{i+1}-x_i}}f+
\frac{1+y\,e^{x_{i}-x_{i+1}}}{1-e^{x_{i}-x_{i+1}}} s_if\,,
&
(\TL_i+id)(\TL_i+y \,id)=0.
\end{matrix}\end{equation*}
Note that we have corrected the sign in the formula  \eqref{partial}
$$\beta^{\rm coh}_i=-\partial_i\,.$$
The above operations act on the corresponding classes of $\B$-orbits of 2-nilpotent matrices belonging to $H^*_\T(\End(\C^m))$,  
$K_\T(\End(\C^m))$ or $K_\T(\End(\C^m))[y]$. In the next section we introduce the elliptic functions on which the calculus of elliptic classes is based.

\section{Theta, $\FF$ and $\delta$ functions}\label{delata}
Let us fix notation and describe the basic properties of the elliptic functions which we will use.
We start with the 
Jacobi theta function. For $x,\, \tau\in \C$,  ${\rm im}(\tau)>0$
Let $q=e^{2\pi\ii\tau}$.
\begin{equation*}\label{Jacobi-prod-for}
\theta_\tau(x)=2q^{\frac18} \sin (\pi x)
\prod_{n=1}^{\infty}(1-q^{n})
 \big(1-q^{n} e^{2\pi\ii x}\big)\big(1-q^{n}/ e^{2\pi\ii x}\big)\,.
\end{equation*}
Our theta function differs from the classical one by a constant factor, which will cancel out in further considerations. Moreover, in some sources the argument $x$ is re-scaled. We have quasi-periodicity relations 
$$\theta_\tau(x+1)=-\theta_\tau(x)\,,$$
\begin{equation*}\label{deltatau0}\theta_\tau(x+\tau)=-q^{-1/2}\, e^{-2\pi\ii x}\, \theta_\tau(x)\,.\end{equation*}
\medskip

The function $\FF_\tau(x,y)$ is defined by the formula
$$\FF_\tau(x,y)=\frac{\theta'_\tau(0)\,\theta_\tau(x+y)}{\theta_\tau(x)\,\theta_\tau(y)}$$
 It is a meromorphic function on $\C^2$. The argument $\tau$ is treated as a parameter. Obviously
 $$\FF_\tau(x,y)=\FF_\tau(y,x)$$
 and
\begin{equation}\label{antysymetria}\FF_\tau(-x,-y)=-\FF_\tau(x,y)\,,\end{equation}
The quasi-periodicity relations take the following form:
\begin{equation*}\label{per1}\FF_\tau(x+1,y)=\FF_\tau(x,y)\,,\end{equation*}
\begin{equation*}\label{per2}\FF_\tau(x+\tau,y)=e^{-2\pi\ii y}\, \FF_\tau(x,y)\,.\end{equation*}
 We will use the function $\delta$ satisfying
\begin{equation}\label{esponential}\FF_\tau(x,y)=\delta_\tau(\ee^{2\pi ix},\ee^{2\pi iy})\,.\end{equation}
since $\FF_\tau$ descends to a function on $\C^*\times\C^*$.
The function $\delta_\tau$ depends also on $q=e^{2\pi \iii \tau}$, but we keep it fixed and do not indicate $\tau$ in the notation. We can treat $\delta$ as an element of
$\Z(a,b)[[q]]$. In fact except from the coefficient of $q^0$ the remaining coefficients of the $q$--expansion are Laurent polynomials in $a$ and $b$. 
The expansion is of the form 
\begin{align}\label{expansion}\delta(x,y)&=\frac{1-x^{-1}y^{-1}}{(1-x^{-1})(1-y^{-1})}+q(x^{-1}y^{-1}- xy)+q^2(x^{-2}y^{-1}+x^{-1}y^{-2}- x^2y- xy^2)+\dots\\
&=\frac{1-x^{-1}y^{-1}}{(1-x^{-1})(1-y^{-1})}+\sum_{n=1}^\infty q^n\sum_{k\ell=n}(x^{-k}y^{-\ell}-x^ky^\ell)\,,\end{align}
see \cite[\S3]{Zag}.
The properties  of the function $\delta_\tau$ are described in \cite{MiWe}.  
The modular properties (with respect to the transformation $\tau\mapsto -1/\tau$) will not play a significant role in the rest of the paper.

\section{The algebra generated by $\delta$-functions and  { grading} by quadratic forms}
\label{bundletype}
Let $\tau\in \C$ with $\operatorname{Im}(\tau)>0$, and let $E_\tau=\C/\langle 1,\tau\rangle$ be the corresponding elliptic curve. Fix natural numbers $m,r\in \N$. We consider the subalgebra of meromorphic functions on $\C^{m+1}\times \C^{r+1}$ generated by the functions
$$
F(x,y)
\quad\text{for}\quad \xb \in \C^{m+1},\quad \mub\in \C^{r+1}\,.
$$
We will use exponential variables since the function $F$ admits the presentation \eqref{esponential}.   The expansion \eqref{expansion} has integral coefficients, but we do not exploit here the integral structure.
\medskip

Throughout the paper, the symbol $\delta(\xb,\mub)$ will be treated mostly formally. The relations required for our purposes are precisely those given in \S\ref{dowodwarkocza}, namely those encapsulated in Theorem \ref{identities} (the relation \eqref{fourterm}, the following blow-up and quadratic relations) as well as the flip relation \eqref{flipkonkretnie}. More precisely, we consider the algebra generated by symbols of the form $\delta(\ell_1(\xb),\ell_2(\mub))$, where $\xb=(u,x_1,\dots,x_m)\in (\C^*)^{m+1}$, $\mub=(h,\mu_1,\dots,\mu_r)\in (\C^*)^{r+1}$, and $\ell_1$, $\ell_2$ are integral linear forms on $\C^{m+1}$ and $\C^{r+1}$, respectively, written multiplicatively. We shall systematically use multiplicative notation for both $\xb$ and $\mub$, since no ambiguity will arise from this convention.

\begin{definition}\label{defAA} The subalgebra of the algebra of the meromorphic functions on 
$$(\tau,\xb,\mub)\in\C_{{\rm Im}(z)>0}\times(\C^*)^{m+1}\times(\C^*)^{r+1}$$ 
generated by the functions $\delta(\ell_1(\xb),\ell_2(\mub))=\delta_\tau(\ell_1(\xb),\ell_2(\mub))$ will be denoted
by $\AA_{m,r}$. 
\end{definition}

A definition of the Demazure-type operators acting on $\AA_{m,r}$, motivated by geometry, will be given in \S\ref{defDemazure}.
The operators $\Cc_i^\mu$ depend on the parameter $\mu=(\mu_1,\mu_2,\dots,\mu_r)$.
Although values $\Cc_i^\mu(f)$ are well-defined for every $f\in \AA_{m,r}$, $\mu\not\equiv 1\mod q^{\mathbb Z}$, they have a geometric significance when $f$ is {\it pure} and $\mu$ is related to the type of $f$. The notion of purity is defined below.

\medskip

Meromorphic functions on $\C^{m+r+2}$ can be interpreted as sections of  line bundles on a product of elliptic curves, provided that they enjoy certain quasi-periodic relations. 
The theta function $\theta_\tau$ defines a global section of a line bundle on $E_\tau$ whose vanishing locus is precisely the origin.
The algebra $\AA_{m,r}$ may be naturally viewed as a subspace of
$$
\bigoplus_{L\in {\rm Pic}(E_\tau^{m+r+2})} H^0(L)
$$
contained in the space spanned by the sections which correspond to Laurent polynomials  in theta functions  in integral combination of variables.

\medskip

We recall the correspondence between quasi-periodic functions and sections of a bundle.
Let $N=m+r+2$. Suppose $Q$ is a symmetric matrix   $N\times N$ with integer entries. It defines a line bundle over the $N$-fold product of a fixed elliptic curve $E=\C/\langle 1,\tau\rangle$. This bundle is the quotient of $\C^N$ by an action of $\Z^N\oplus\Z^N$. The action is defined by
$$(k,\ell)\cdot (z,v)=\left(z+k+\ell\tau\;,\;(-1)^{\langle k, k\rangle}(-q^{1/2})^{\langle\ell, \ell\rangle}e^{-2\pi\iii {\langle\ell , z\rangle}}v\right)$$
Here the product $\langle a,b\rangle=a^TQb\in \mathbb Z$ is defined for integral vectors  and extended linearly to $\C^N$.
According to our convention  $q^{1/2}=\ee^{\pi \iii\tau}$.
With this interpretation the theta function is a section of the line bundle over $E$ associated to the  $1\times 1$ matrix $[\,1\,]$.  
Quasi-periodic functions are the functions defined on $\C^N$ and descend to sections of the line bundle defined above.

\medskip

Instead of  symmetric matrices $Q$ it is more convenient to perform calculus of quadratic forms. The matrix $[\,1\,]$ corresponds to the quadratic form $\frac12x^2$. We will say that $\theta(x)$ is of the type $\frac12x^2$ and write 
$$\type(\theta(x))=\tfrac12 x^2\,.$$
Products of functions result in addition of quadratic forms and quotients give differences. For example the function $\FF(x,y)= \frac{\theta'(0)\theta(x+y)}{\theta(x)\theta(y)}$ is a section of the bundle associated to the quadratic form $$\tfrac12\big((x+y)^2-x^2-y^2\big)=xy\,.$$
We write $\type(\delta(x,y))=xy$. 
\medskip

\begin{definition}\label{defpure}We  say that a Laurent polynomials in theta functions applied to integral combination of variables is {\it pure}  if all summands have the same type. A pure function defines a  section of the associated line bundle.
\end{definition}

\medskip

{\bf Important convention:} When a combination of variables appears as an argument of the $\delta$ function we use the multiplicative notation. This is to save space and to agree with the literature. 
In formulas involving quadratic forms we use additive notation. 
Multiplication of arguments of $\delta$ corresponds to addition of variables in types, for example
$$\type\Big(\delta\big(\frac{x_{i+1}}{x_i},\frac{\mu_{i+1}}{\mu_i}\big)\Big)=(x_{i+1}-x_i)(\mu_{i+1}-\mu_i)\,.$$

\begin{example}\rm 
(compare Example \ref{Schred}) The function
$$\delta (\tfrac{x_1}{x_2},h)
   \delta
   (\tfrac{x_2}{y_2},h)
   \delta (\tfrac{x_2}{y_1},\mu
   _1) \delta
   (\tfrac{x_1}{y_2},\mu
   _2)+\delta
   (\tfrac{x_2}{x_1},\tfrac{\mu
   _1}{\mu _2}) \delta
   (\tfrac{x_1}{y_2},h)
   \delta (\tfrac{x_1}{y_1},\mu
   _1) \delta
   (\tfrac{x_2}{y_2},\mu _2)
$$
is pure. Its type is equal to
\begin{multline*} \left(x_1-x_2\right)h+
   \left(x_2-y_2\right)h+
   \left(x_2-y_1\right)\mu _1+
   \left(x_1-y_2\right)\mu _2=\\
   \left(x_2-x_1\right)\left(\mu _1-\mu _2\right)+
    \left(x_1-y_2\right)h+
  \left(x_1-y_1\right)\mu _1+
   \left(x_2-y_2\right)\mu _1\,.
\end{multline*}
Therefore, the sum makes sense when interpreted as sections of the same line bundle.
\end{example}
The pure functions belonging to the algebra $\AA_{m,r}$ have very special quadratic forms: they are linear in $\xb$-variables and linear in $\mub$ variables, since $\AA_{m,r}$ is generated by the functions $\delta(\ell(\xb),\ell'(\mub))$, where $\ell$ and $\ell'$ are linear functionals:
$$\type\!\Big(\smash\prod_{i=1}^k\delta(\ell_i(\xb),\ell_i'(\mub))\Big)=\sum_{i=1}^k\ell_i(\xb)\ell_i'(\mub)\,.$$

\begin{remark}\rm The analysis of quadratic forms plays a crucial role in the work on {\it bow varieries} and their stable envelopes, \cite[\S3.3]{BottaRi}.\end{remark}

\section{Elliptic characteristic classes of Schubert varieties }\label{dlaGbyB}
The torus equivariant K-theory of the flag variety $\Fl_n=\GL_n(\C)/\B_n$ admits the Borel presentation 
\begin{equation}\label{K-flag}K_\T(\Fl_n)\simeq{ R(\T)}\otimes_{R(\T)^W}R(\T)\,,\end{equation}
where $R(\T)$ stands for the representation ring, which is identified with $$K_\T(pt)=\Z[x_1^{\pm1},x_2^{\pm1},\dots,x_n^{\pm1}]\,.$$
The variables $x_i$ are called equivariant variables, and in the presentation \eqref{K-flag} they correspond to the first factor. The variables of the second factor are denoted by $y_i$ and by the isomorphism \eqref{K-flag} they are sent to the classes of the tautological line bundles. { We} call them topological variables. 
For the definition of elliptic classes we need also to consider a family of boundary conditiotions depending on line bundles.
In \cite{RiWe1} we have decided to work with the following model of elliptic equivariant cohomology extended by the additional parameters:
$$\text{Meromorphic maps}:\mathbb H\times \T^\vee\times \C^*\longrightarrow K_\T(\Fl_n)\otimes\C\,.$$
Here $\mathbb H$ was the upper half-plane parametrizing elliptic curves, the dual torus $\T^\vee$ with variables $\mu_i$ was parametrizing the boundary conditions and $h\in\C^*$ was an additional parameter. Therefore,
the elliptic classes of Schubert varieties in the complete flag variety  are expressions in:

\begin{itemize}
\item $\g_i$, $i=1,2,\dots n$ --- equivariant variables,   coming from the coefficient ring, i.e.~the equivariant elliptic cohomology of a point. The variables $\g_i$ are characters of the maximal torus in $\T\subset \GL_n(\C)$. 
\item $\li_i$, $i=1,2,\dots n$ --- the classes of the tautological line bundles. The line bundles on the complete flag variety
correspond to characters as well, hence $\li_i$ is another set of basis characters of the maximal torus.
\item $\mu_i$, $i=1,2,\dots n$ --- 
the  functionals on $\T^\vee\simeq {\rm Pic}(\Fl_n)\otimes\C/\Z$, thus, can be understood as cocharacters. 
\item $h$ --- an auxiliary  variable.
\end{itemize}

Our construction allows considering combinations of the variables $\mu_i$ with complex coefficients although in the original definition of the boundary divisor, given by Borisov and Libgober \cite{BoLi}, the coefficients were rational.
The corresponding Demazure-type operators have a parameter built in. The localized\footnote{The quotients of the elliptic class by the Euler classes of the tangent bundle.}  elliptic classes 
are subject to the following relations 
\begin{enumerate}[(i)]
\item Bott-Samelson recursion: if $\ell(ws_i)>\ell(w)$ 
\begin{equation}\label{BSrecursion}
\Elclas(X_{w s_i }) =
\delta( \tfrac{\li_{i+1}}{\li_i},\tfrac{\mu_{i+1}}{\mu_i})\cdot s_i^\mu\Elclas(X_{w})  
+
\delta(\tfrac{\li_{i+1}}{\li_i},h)\cdot s_i^\li s_i^\mu\Elclas(X_{w}).
\end{equation}
\item R-matrix recursion: if $\ell(s_iw)>\ell(w)$
\begin{equation}\label{Rrecursion}
\Elclas(X_{s_i w}) =
\delta(\tfrac{\g_{i+1}}{\g_{i}},\tfrac{\mu_{w^{-1}(i+1)}}{\mu_{w^{-1}(i)}}) \cdot\Elclas(X_{w}) 
+
\delta(\tfrac{\g_i}{\g_{t+1}},h)\cdot s_i^\g \Elclas(X_w)
\,,\end{equation}
\end{enumerate}
Here the transposition $s^\g_i$ acts on $\g$ variables, $s^\li_i$ acts on $\li$-variables, $s^\mu_i$ acts on $\mu$ variables. See \cite[\S11]{RiWe1}, \cite[\S3]{RiWe2}.  
\medskip

 In general, the Borisov-Libgober elliptic class depends on the multiplicities imposed on the boundary divisor, and the variables $\mu_i$ determine  the boundary divisor for each Schubert variety. Moreover , the parameters $\mu$ have a good global interpretation.  In the case of the homogeneous space $G/B$ the  choice of the boundary divisor can be made uniformly for all the Schubert varieties. We recall the construction of \cite[\S4]{RiWe1}. Let  $L_\lambda\in {\rm Pic}(G/B)$ be a  fractional line bundle on $G/B$ determined by $\lambda\in \mathfrak t^*_{\mathbb Z}$. Then $L_\lambda|_{X_w}$ defines a divisor supported by the boundary $\partial X_w$. The coefficients $a_{w,v}$ of its components $X_v\subset \partial X_w$ are given by the Chevalley formula  $${~~} c_1(L_\lambda)\cap[X_w]=\sum_v a_{w,v} [X_v]\in H_{2(\dim X_w-1)}(X_w)\,,\quad \text{ where }\cap\colon H^*(X_w)\otimes H_*(X_w)\to H_*(X_w),$$ 
\cite{Chevalley},\cite[Prop.
1.4.5]{BrionFlag}. The recursion for the elliptic classes has built in the recursion of the coefficients following from the Kempf formula \cite[Lemma 3]{Kem76}. Therefore, the variables $\mu_i=h^{1-\lambda_i}$ can be interpreted as functions on $\mathfrak t^*$, while the equivariant variables $\li_i$  are characters, i.e. functions on  $\mathfrak t$.
\medskip

There is certain asymmetry in the formulas (\ref{BSrecursion}--\ref{Rrecursion}), which is caused by different roles of variables.  This asymmetry disappears when we do not divide $\GL_n$ by the Borel group. The full algebra of operations can be encapsulated in the construction described below in Section \S\ref{rozdzialalgebra}.

\section{Elliptic classes of 2-nilpotent orbits -- the beginning}\label{beginning}
The elliptic characteristic classes of Borisov and Libgober are defined for certain class of singular algebraic varieties. The construction is generalized
to pairs $(X,D)$, where $D$ is a divisor and $X$ is smooth away from the support of the divisor $D$. See \cite{BoLi}, \cite[\S2]{RiWe1} for the definition. The construction involves a resolution of singularities and the pull-back of the relative canonical divisor $K_X+D$. In particular, it is assumed that $K_X+D$ is a $\Q$-Cartier divisor.\medskip

The appearance of the canonical divisor $K_X$ may be somewhat surprising, but there are at least two reasons why it plays a role. First, this definition of a characteristic class applies to a wide class of singularities, and it can be shown (see \cite{Totaro, AW4}) that any characteristic class invariant under the Atiyah flop is, in essence, obtained via the Borisov-Libgober construction. Second, in birational geometry, the behavior of the canonical divisor and its multiplicities along the exceptional divisors of a resolution plays a crucial role; see {\cite[Lecture 6]{CKM}.}

\medskip

 We do not have enough information about the canonical divisors of $\B$-orbit closures of  2-nilpotent matrices, therefore we chose a slightly different way. 
Our geometric objects are  the closures of the orbits
$$X_{m,r}^w=\overline{\B_mw \cdot N^{\min}_{m,r}}\,,$$
where $w\in\Spe_m$ is a permutation and $N^{\min}_{m,r}$ is the minimal matrix of the size $m$ and rank $r$
{\setlength\arraycolsep{1pt}\begin{equation}\label{ener} N^{\min}_{m,r}=
\hbox{\tiny $\begin{bmatrix}
0&0&\dots&0&1&0&\dots&0\\
0&0&\dots&0&0&1&\dots&0\\[-5pt]
&&&&\vdots&&\ddots\\[-3pt]
0&0&\dots&0&0&0&\dots&1\\[-5pt]
&&&&\vdots&&\\[-3pt]
0&0&\dots&0&0&0&\dots&0\end{bmatrix}$}
\,,\quad 2r\leq m\end{equation}}
\noindent (there are $r$ entries with 1 on a parallel to the diagonal).   
Suppose $$\dim(X_{m,r}^w)=\ell(w)+\dim(X^{\id}_{m,r})\,,$$ in that case the permutation $w$ has the minimal possible length among the permutations transforming  $X^{\id}_{m,r}$ to $X_{m,r}^w$.
Then according to \cite{BePe} the variety $X_{m,r}^w$ admits a desingularization of the form 
\begin{equation}\label{BePerez}Z\strut^{\underline w}_{m,r}=P_{i_1}\times_\B  P_{i_2}\times_\B \dots \times_\B  P_{i_\ell}\times_\B  X^{\id}_{m,r}\,,\end{equation}
where $\underline w=s_{i_1}s_{i_2}\dots s_{i_\ell} $ is a reduced word representing $w$, and $P_i\subset \GL_m$ is the minimal parabolic subgroup corresponding to simple reflection $s_i$.

We will consider the push forward of the elliptic class of $(Z\strut_{m,r}^{\underline w},D)$ with $D$ defined in the (unique) way guaranteeing  the push-forward is pure. 
The resulting classes belong to the equivariant elliptic cohomology of $\Hom(\C^m,\C^m)$ with additional parameters, i.e. to the algebra $\AA_{m,r}$, a subalgebra of the algebra of meromorphic functions in $\tau$, $\xb$ and $\mub$  defined in section \ref{bundletype}.
We will prove algebraically that the resulting class does not depend on the reduced word. Moreover, our construction allows considering   non-reduced words and  permutations, not necessarily defining a birational map $Z\strut_{m,r}^{\underline w}\to  X\strut_{m,r}^{ w}$. 
It allows treating elliptic classes of link patterns as effects of an action of the algebra  defined in \S\ref{rozdzialalgebra}.
\medskip

Let us have a look at beginning of the induction.  
 The variety $X^{\min}_{m,r}=X_{m,r}^\id$ is just a vector subspace of $\End(\C^m)$ described by
$$a_{i,j}=0\quad\text{for }i=1,2,\dots m\,,~~1\leq j<m-r+i\,,$$
where $a_{i,j}$ are the entries of the matrix.
The boundary divisor of $\B_m\cdot N_{m,r}$ is given by
{\setlength\arraycolsep{1pt}
$$\prod_{i=1}^r a_{i,m-r+i}=0\,,\qquad\hbox{\tiny 
$\begin{bmatrix}
0&0&\dots&0&\bullet&*&\dots&*\\
0&0&\dots&0&0&\bullet&\dots&*\\[-5pt]
&&&&\vdots&&\ddots\\[-3pt]
0&0&\dots&0&0&0&\dots&\bullet\\[-5pt]
&&&&\vdots&&\\[-3pt]
0&0&\dots&0&0&0&\dots&0\end{bmatrix}$}
\,.
$$}
We attach a generic multiplicity $\lambda_i$   to each  component of the boundary divisor $$D_i=\{a_{i,m-r+i}=0\}\,.$$  
In the Borisov-Libgober definition of the elliptic class the multiplicities appear only  combined with $h$ as $h^{1-\text{multiplicity}}$. Since we treat the variable $\lambda_i$ as free variables we introduce 
\begin{equation}\label{muintro}\mu_i=h^{1-\lambda_i}\end{equation} and in further formulas $\lambda_i$ will not appear. The arguments of the $\delta$-function cannot be equal to $1$, because at $\mub=1$ the function $\delta(\xb,\mub)$ has a pole. Our inductive construction leads to factors that have integral combinations of multiplicities. It is formally correct to consider the arguments $\mu_i$ as independent variables. Determining which actual combinations appear in the formula for a given pattern requires separate analysis. By  an algorithm for the resolution of singularities of $2$-nilpotent matrices, we can find  that poles can are at hypersurfaces $\mu_i\equiv \mu_j\in \C^* \mod q^{\mathbb Z} $.
\medskip

In section \ref{dlaGbyB} summarizing the results about elliptic classes of Schubert varieties in $\GL_n/\B_n$ the variables $\mu_i$ denoted basis cocharacters of the maximal torus. Now we are in a more general situation. The variables $\mu_i$ are the  multiplicities at the boundary of the minimal rank $r$ 2-nilpotent orbit. However, we attempt to maintain consistent notation with the formula from the previous case, which corresponds to $m=2n$ and $r=n$.
\medskip

The $\B_m$ orbits are preserved by the torus $\T=(\C^*)^m\times \C^*$. The first factor is the maximal torus of $\GL_m$ acting by conjugation on matrices, the second factor acts by scalar multiplication.
The corresponding characters are denoted by $x_i$, $i=1,2,\dots,m$ and $u$.
We are at the position to apply directly the definition of Borisov and Libgober, or better use the formula \cite[formula (5)]{RiWe1}. No desingularization is needed here. 
We obtain 
\begin{equation}\label{poczatek}\Elclas(X^{\min}_{m,r},\sum_{i=1}^r\lambda_iD_i)=\prod _{i=1}^r\left(\delta(u\tfrac{x_i}{x_{m-r+i}},\mu_i)\cdot\prod_{j=m-r+i+1}^m \delta(u\tfrac{x_i}{x_j},h)\right)\,.\end{equation} 
This is our starting point. The remaining elliptic classes of Borel orbits are well-defined due to the property of the Hecke algebra action, in particular the braid relations \eqref{ellbraid} and the flip relation \eqref{flip}. This is shown in \S\ref{niezaleznosc}. We avoid  analysis of the canonical divisor.

\section{Action of elliptic Demazure operations on functions}
\label{defDemazure}
\label{rozdzialalgebra}
We construct elliptic Demazure operators acting on the algebra $\AA_{m,r}$ defined in Definition \ref{defAA}.
The unified elliptic Demazure operations are lifts and extensions of those defined in \cite{RiWe1}.
The action on a meromorphic function $f\in \AA_{m,r}$ is defined by the formula
\begin{equation}\label{elldemazur}\Cc_i^\mu(f)=\delta\big(\tfrac{x_{i+1}}{x_i},\mu\big)f+\delta\big(\tfrac{x_i}{x_{i+1}},h\big)s_if\,,\qquad i=1,2,\dots,m-1.\end{equation}
Here $s_i$ acts only on $x$-variables.
The operations $\Cc^\mu_i$ satisfy braid relations with parameters
\begin{equation}\label{ellbraid}\Cc_i^{\mu}\;\Cc_{i+1}^{\mu \nu}\;\Cc_i^{\nu}=
\Cc_{i+1}^{\nu}\;\Cc_i^{\mu \nu}\;\Cc_{i+1}^{\mu}\,.\end{equation}
The braid relation (Yang-Baxter equation with spectral parameter, 
see e.g.~\cite[\S12.1, p 380]{ChariPressley} or \cite{JKOS} for other variants.) can be rewritten as 
$$\Cc_i^{\beta/\gamma}\;\Cc_{i+1}^{\alpha/\gamma}\;\Cc_i^{\alpha/\beta}=
\Cc_{i+1}^{\alpha/\beta}\;\Cc_i^{\alpha/\gamma}\;\Cc_{i+1}^{\beta/\gamma}$$
and represented by the following picture 
$$\begin{tikzcd}[nodes={inner sep=0pt}] 
     \alpha \arrow[-]{dr}{\hskip-6pt\frac\alpha\beta}
    & \beta\arrow[-]{dl}
    & \gamma\arrow[-]{d}
        \\
     \beta\arrow[-]{d} 
    & \alpha\arrow[-]{dr} {\hskip-6pt\frac\alpha\gamma}
    & \gamma\arrow[-]{dl}
     \\
   \beta \arrow[-]{dr} {\hskip-6pt\frac\beta\gamma}
    & \gamma\arrow[-]{dl} 
    & \alpha\arrow[-]{d}
    \\
      \gamma & \beta & \alpha  \\
\end{tikzcd}\hskip20pt\begin{matrix}=\\ \\ \end{matrix}\hskip20pt\begin{tikzcd}[nodes={inner sep=0pt}] 
    \alpha\arrow[-]{d} 
    &\beta \arrow[-]{dr}{\hskip-6pt\frac\beta\gamma} 
    & \gamma\arrow[-]{dl}
    \\
    \alpha\arrow[-]{dr}{\hskip-6pt\frac\alpha\gamma} 
    & \gamma\arrow[-]{dl} 
    & \beta\arrow[-]{d}
    \\
    \gamma\arrow[-]{d} 
    &\alpha \arrow[-]{dr}{\hskip-6pt\frac\alpha\beta} 
    & \beta\arrow[-]{dl}
    \\
  \gamma   
    & \beta&
        \alpha&\hskip-30pt.   \\
\end{tikzcd}
$$
The operations $\Cc_i^\mu$ satisfy the following quadratic relation
$$\Cc_i^{\mu}\;\Cc_i^{1/\mu}=\delta(h,\mu)\delta(h,1/\mu)\id\,.$$
After normalization 
\begin{equation}\label{normalizacja}\Cbar_i^{\mu}=\tfrac1{\delta(\mu,h)}\Cc_i^{\mu}\end{equation}
we obtain
$$\Cbar_i^{\mu}\;\Cbar_i^{1/\mu}=\id\,.$$

It should be noted that the operations $\beta_i^{\rm coh}$ and $\TL_i^{\rm coh}$ of the table in \S\ref{tabela} act on the  polynomial ring, which is identified with $H^*_\T(pt)$, the operations $\beta_i$ and $\TL_i$ act on the ring of Laurent polynomials (extended by the variable $y$) -- which is identified with $K_\T(pt)[y]$. The natural domain of the operations $\Cbar^\mu_i$ is the ring of meromorphic functions in $(x,\mu,h)$ which descend to sections of  line bundles over the product of elliptic curve $E$. Additionally we need a variable $u$ which is not involved in the definition of $\Cc_i^{\mu}$.
\medskip

The proofs are given in Section \S\ref{dowodwarkocza}.

\begin{remark}[Weyl group representation]\rm 
Let us restrict our attention to the case $m=2n$, $r=n$. We identify the variables $x_{n+i}=\li_i$ for $i\leq n$ with Chern classes of the tautological bundles on the flag variety, while
characters of the maximal torus in $\GL_n$ remain to be denoted by $x_i$ for $i\leq n$. 
Let $\T^\vme$ be the dual torus, $\mu_i\colon\T^\vme\to\C^*$ the corresponding cocharacters. In \cite{RiWe1, RiWe2} we have defined the operations on $Mero(\T^2\times\T^\vme\times\C^*)$ 
$$\Ccal_i=\Cc_i^{\mu_i^\vme}\circ s^\mu_i\,,
$$
where $s^\mu_i$ is the reflection acting on $\T^\vme$ inducing an action  on the functions on  $\T^\vme$. 
Explicitly:
$$\Ccal_i(f)(z,\gamma,\mu)=\delta(\tfrac{\gamma_{i+1}}{\gamma_i},\tfrac{\mu_{i+1}}{\mu_i})\,f(z,\gamma,s_i(\mu))+\delta(\tfrac{\gamma_{i}}{\gamma_{i+1}},h)\,f(z,s_i(\gamma),s_i(\mu))$$ 
It is shown in \cite[Theorem 5.1]{RiWe1} that $$\Ccal_i\circ\Ccal_i=\delta(h,\tfrac{\mu_{i+1}}{\mu_i})\delta(h,\tfrac{\mu_{i}}{\mu_{i+1}})\,\id$$
after reduction to elliptic cohomology of the flag variety. In fact the relation holds on the level of functions in $x_i$, $\li_i$ and $\mu_i$. This is a special case of Theorem \ref{identities}.
After normalization we obtain an action of the permutation group, as in  \cite[Proposition 4.11]{ZhZh1}.
\end{remark}

\section{The main results}
For each labelled link pattern we will construct a meromorphic function in variables $x\in (\C^*)^{m}$, $\mu\in (\C^*)^r$ and $u,h\in \C^*$
$$\text{labelled  link pattern }\pa~~\mapsto~~ \Ep(\pa)\in \AA_{m,r}$$
with the following properties:
\begin{enumerate}[(i)]
\item the function $\Ep(\pa)$ has a pure type, i.e.~it defines a section of a line bundle, as explained in \S\ref{bundletype}),
\item $\Cc_i^\mu(\Ep(\pa))=\Ep(s_i\pa)$ for $\mu\in(\C^*)^r$ uniquely determined by the purity condition
\item for the minimal orbit of rank $r$ the class
$\Ep(\pa^{\min}_{m,r})$ is given by the formula \eqref{poczatek},
\item if $m=2n$ and the link pattern represents a permutation $\sigma\in \Spe_n$, then (a normalization of) $\Ep(\pa)$ represents the equivariant elliptic class of the Schubert variety $X_\sigma$,
\item if $m=2n-1$ and the link pattern represents an injective map $$\{1,2,\dots,n-1\}\to \{1,2,\dots,n\}\,,$$ then the elliptic class is equal to the elliptic weight function of \cite{RTV} (up to a change of variables and normalization),
\item $\Ep(\pa)$ is  the push forward from Bott-Samelson resolution of the elliptic class in the sense of Borisov-Libgober for a unique choice of the boundary divisor guaranteeing purity.
\end{enumerate}
The conditions (i)--(iii) determine the function $\Ep(\pa)$.
\medskip

It should be noted that a change of arrow labels  results in permutation of $\mu$-variables in $\Ep(\pa)$. 
 The key argument of the proof of independence of the choice of a resolution is based on an application  of the reduced operations $\Cbar_i^\mu$, for the right choice of $\mu$, forced by the purity condition. We obtain a $\Spe_m$-action, hence we do not have to focus on the reduced words representing permutations. Moreover, the change of purity types is well controlled:
 $$\type(\Ep(s_i\pa))=s_i(\type(\Ep(\pa))-h\rho_m)+h\rho_m\,,$$
 see Theorem \ref{wyznaczenienu}.
Here $\rho_m$ is given by \eqref{halfsum}.
 This allows to write a ready to use formula for the type of $\Ep(w(\pa^{\min}_{m,r}))$ for $w\in \Spe_m$ in combinatorial terms.
 \medskip
 
We apply   Borisov and Libgober results to show that the  identity 
\begin{equation}\label{flip}s_k^\mu\Cc^{\mu_k/\mu_{k+1}}_k(\Ep(\pa^{\min}_{m,r}))=\Cc^{\mu_k/\mu_{k+1}}_{m-r+k}(\Ep(\pa^{\min}_{m,r}))\qquad \text{for }1\leq k<r\,.\end{equation}
holds.
We call the resulting identity of elliptic functions {\it the flip relation}. 
 Although this identity can be read from the four-term relation \cite[eq.~(2.7)]{RTV} we wish to give its geometric proof.
 For $m=4$, $r=2$ the identity is illustrated by the transformation of link patterns
{\def\bullet{\cdot}
\begin{equation}\label{przyklad42}
\xymatrix@-1.4pc{
\bullet\ar@{<-}@/^1.5pc/^{\mu_1}[rr]
&
\bullet\ar@{<-}@/^1.5pc/^{\mu_2}[rr]
&
\bullet
&	
\bullet&=&\\
\bullet\ar@{-}[ur]
&
\bullet\ar@{-}[ul]
&\bullet\ar@{-}[u]
&
\bullet\ar@{-}[u]
&&
\bullet\ar@{<-}@/^3.5pc/^{\mu_2}[rrr]
&
\bullet\ar@{<-}@/^2pc/^{\mu_1}[r]
&
\bullet
&	
\bullet}
\hskip30pt
\begin{matrix}\stackrel{s_1^\mu}{\leftarrow\hskip-0.1em \rightarrow}\\ \end{matrix}
\hskip30pt
\xymatrix@-1.4pc{
\bullet\ar@{<-}@/^1.5pc/^{\mu_1}[rr]
&
\bullet\ar@{<-}@/^1.5pc/^{\mu_2}[rr]
&
\bullet
&	
\bullet&=&\\
\bullet\ar@{-}[u]
&
\bullet\ar@{-}[u]
&\bullet\ar@{-}[ur]
&
\bullet\ar@{-}[ul]
&&
\bullet\ar@{<-}@/^3.5pc/^{\mu_1}[rrr]
&
\bullet\ar@{<-}@/^2pc/^{\mu_2}[r]
&
\bullet
&	
\bullet}
\end{equation}
}
Here $s_1^\mu$ is the transposition of labels of the link pattern.
\medskip
The flip identity for $m=4$, $r=2$ is proven in \S\ref{basicexample}. It implies  identities  in higher dimensions, see the proof of Theorem \ref{niezalezy}. 
\medskip

We also need the Fay trisecant relation in the form \cite[Thm.~5.3]{FRV} or \cite[\S4.1]{MiWe}. 

\section{Verification of the braid and quadratic identities}\label{dowodwarkocza}
Let us fix $r\geq0$ and $m\geq 2r$. The operation  $\Cc^\mu_i$ defined by \eqref{elldemazur} acts on meromorphic functions  on $(\C^*)^{m+r+2}$ (with coordinates $x_i,u,\mu_j,h$, $i=1,2,\dots m$, $j=1,2,\dots, r$). The operation $\Cc^\mu_i$  depends on a character $\mu:(\C^*)^r\to \C$ written as a combination of basic characters $\mu_1,\mu_2,\dots,\mu_r$. 
We assume that $\mu\not\equiv 1\mod q^{\mathbb Z}$, since at $\mu=1$ the function $\delta(x,\mu)$ has a pole.
Obviously $\Cc_i^\mu$ commutes with $\Cc_j^\mu$ if $|i-j|>1$.
We focus  on  $m=3$, the general case follows.
\begin{theorem}\label{identities} The operations $\Cc_i^\mu$ satisfy the twisted braid relations
\begin{equation}\label{braid}\Cc_1^{\nu}\;\Cc_2^{\mu\nu}\;\Cc_1^{\mu}=
\Cc_2^{\mu}\;\Cc_1^{\mu\nu}\;\Cc_2^{\nu}\,.\end{equation}
Moreover,
\begin{equation}\label{nonredkwadrat}\Cc_i^{\mu}\;\Cc_i^{1/\mu}=\delta(h,\mu)\delta(h,1/\mu)\,.\end{equation}
\end{theorem}

\begin{proof} We adjust the notation to the one used in \cite{RiWe1}.
We set $\nu=\frac{\mu_2}{\mu_1}$ and $\mu=\frac{\mu_3}{\mu_2}$ and we rewrite the braid relation:
$$\Cc_1^{\mu_2/\mu_1}\;\Cc_2^{\mu_3/\mu_1}\;\Cc_1^{\mu_3/\mu_2}=
\Cc_2^{\mu_3/\mu_2}\;\Cc_1^{\mu_3/\mu_1}\;\Cc_2^{\mu_2/\mu_1}\,.$$
To check that relation we look at the coefficients of $f(x_i,x_j,x_k)$. We have to prove a number of identities.
Comparing the coefficients of $f\left(x_1,x_2,x_3\right)$ we have to show
{\small
\begin{multline} \label{fourterm}\delta \left(\frac{x_1}{x_2},h\right)
   \delta \left(\frac{x_2}{x_1},h\right) \delta
   \left(\frac{x_3}{x_1},\frac{\mu _3}{\mu _1}\right)+\delta
   \left(\frac{x_2}{x_1},\frac{\mu _2}{\mu _1}\right) \delta
   \left(\frac{x_2}{x_1},\frac{\mu _3}{\mu _2}\right) \delta
   \left(\frac{x_3}{x_2},\frac{\mu _3}{\mu _1}\right) =\\
   =\delta
   \left(\frac{x_2}{x_3},h\right) \delta \left(\frac{x_3}{x_2},h\right)
   \delta \left(\frac{x_3}{x_1},\frac{\mu _3}{\mu _1}\right)+\delta
   \left(\frac{x_2}{x_1},\frac{\mu _3}{\mu _1}\right) \delta
   \left(\frac{x_3}{x_2},\frac{\mu _2}{\mu _1}\right) \delta
   \left(\frac{x_3}{x_2},\frac{\mu _3}{\mu _2}\right). \end{multline}}
   This is exactly the equality  \cite[(32)=(33)]{RiWe1} or \cite[eq.~(2.20)]{RTV}.
The comparison of the remaining coefficients leads to the equations
{\small
$$\begin{array}{cc}
 f\left(x_1,x_3,x_2\right): & \delta \left(\frac{x_2}{x_3},h\right)
   \left(\delta \left(\frac{x_2}{x_1},\frac{\mu _3}{\mu _2}\right)
   \delta \left(\frac{x_3}{x_1},\frac{\mu _2}{\mu _1}\right)-
   \delta \left(\frac{x_2}{x_3},\frac{\mu _3}{\mu _2}\right) 
   \delta \left(\frac{x_3}{x_1},\frac{\mu _3}{\mu _1}\right)-
   \delta \left(\frac{x_2}{x_1},\frac{\mu _3}{\mu _1}\right) 
   \delta \left(\frac{x_3}{x_2},\frac{\mu _2}{\mu _1}\right)\right) =0\,, \\ \\
 f\left(x_2,x_1,x_3\right): & \delta \left(\frac{x_1}{x_2},h\right)
   \left(
   \delta \left(\frac{x_1}{x_2},\frac{\mu _2}{\mu _1}\right)
   \delta \left(\frac{x_3}{x_1},\frac{\mu _3}{\mu _1}\right)+
   \delta \left(\frac{x_2}{x_1},\frac{\mu _3}{\mu _2}\right) 
   \delta \left(\frac{x_3}{x_2},\frac{\mu _3}{\mu _1}\right)-
   \delta \left(\frac{x_3}{x_1},\frac{\mu _3}{\mu _2}\right) 
   \delta \left(\frac{x_3}{x_2},\frac{\mu _2}{\mu _1}\right)
\right) =0\,,\\
 \end{array}
$$}
which follow from the blow-up relation \cite[Ex.~2.10]{RiWe1}.
The reaming coefficients are equal on the nose. 
The quadratic relation is equivalent to:
$$\begin{array}{cl}
f(x_1,x_2):&\delta
   \left(\frac{x_1}{x_2},h\right) \delta
   \left(\frac{x_2}{x_1},h\right)+\delta
   \left(\frac{x_2}{x_1},\frac{1}{\mu }\right) \delta
   \left(\frac{x_2}{x_1},\mu \right) =
 \delta \left(\frac{1}{\mu },h\right) \delta (\mu ,h) \\ \\
f(x_2,x_1):& \delta \left(\frac{x_1}{x_2},h\right) \left(\delta
   \left(\frac{x_1}{x_2},\mu \right)+\delta
   \left(\frac{x_2}{x_1},\frac{1}{\mu }\right)\right)=0 \,.\\
\end{array}$$
The first one again follows from the blow-up relation (see \cite[Ex.~2.10]{RiWe1}, \cite[\S4.1]{MiWe}) and the second one is the easiest, since it follows from anti-symmetry  of $\delta$, see \eqref{antysymetria}.
\end{proof}

It is convenient to normalize the  operations $\Cc^\mu_i$
$$\Cbar_i^\mu=\tfrac1{\delta(\mu,h)}\Cc_i^\mu\,.$$
Then
\begin{equation}\label{redkwadrat}\Cbar_i^{\mu}\;\Cbar_i^{1/\mu}=\id\,.\end{equation}
The normalized operations $\Cbar_i^\mu$ satisfy the braid relation 
\begin{equation}\label{redbraid}\Cbar_1^{\nu}\;\Cbar_2^{\mu\nu}\;\Cbar_1^{\mu}=
\Cbar_2^{\mu}\;\Cbar_1^{\mu\nu}\;\Cbar_2^{\nu}\,.\end{equation}
since 
$$\Cbar_1^{\nu}\;\Cbar_2^{\mu\nu}\;\Cbar_1^{\mu}=\tfrac1{\delta(\nu,h)\delta(\mu\nu,h)\delta(\mu,h)}\Cc_1^{\nu}\;\Cc_2^{\mu\nu}\;\Cc_1^{\mu}$$
and
$$\Cbar_2^{\mu}\;\Cbar_1^{\mu\nu}\;\Cbar_2^{\nu}=\tfrac1{\delta(\mu,h)\delta(\mu\nu,h)\delta(\nu,h)}
\Cc_2^{\mu}\;\Cc_1^{\mu\nu}\;\Cc_2^{\nu}\,.$$

\section{Preserving purity}
We consider meromorphic functions on $(\C^*)^{m+r+2}$ in variables $u=x_0$, $x_i$ ($i=1,2,\dots,m$),  $h=\mu_0$, $\mu_j$ ($j=1,2,\dots,r$),  which belong to $\AA_{m,r}$. Recall that the notion of purity was defined in section \S\ref{bundletype}, Definition \ref{defpure}.  

\begin{lemma}\label{wzor_na_nu}Suppose 
$0\neq f\in \AA_{m+1,r}$ 
is a pure   function of the type 
$\type(f)$, then $\Cc^\mu_i(f)$ is pure 
if and only if
$$\mu=\partial_i(\type(f))-h\,.$$
where
$$\partial_i(\type(f))=\frac{\type(f)-s_i\type(f)}{x_i-x_{i+1}}\,.$$
\end{lemma}
\begin{proof}
The formula \eqref{elldemazur} defining $\Cc_i^\mu(f)$ has two summands. The first one has the type
\begin{equation} \label{typbd}\type\left(\delta\big(\frac{x_{i+1}}{x_i},\mu\big)f\right)=\type(f)+(x_{i+1}-x_i)\mu\end{equation}
and the second summand
\begin{equation}\label{typin}\type\left(\delta\big(\frac{x_{i}}{x_{i+1}},h\big)s_if\right)=s_i\type(f)+(x_{i}-x_{i+1})h\,.\end{equation}
Transforming the equality of \eqref{typbd} and \eqref{typin}
we obtain
$$\mu+h=\frac{\type(f)-s_i\type(f)}{x_{i}-x_{i+1}}=\partial_i(\type(f))\,.$$
\end{proof}

We will say that $\mu\in(\C^*)^m$ is admissible for the operation $\Cc^\mu_i$ 
(or equivalently for $\Cbar^\mu_i$) applied to a pure function $f$ if the result $\Cc^\mu_i(f)$ is pure. By Lemma \ref{wzor_na_nu} $\mu$ is admissible 
(for the particular pure function $f$ and $i\in\{1,2,\dots,m-1\}$)
if and only if $\mu=\partial_i(\type(f))-h$.
\medskip

Let $\rho_m$ be defined by the formula \begin{equation}\label{halfsum}\rho_m=-\sum_{i=1}^m i \,x_i\,.\end{equation}
 We have
$$\partial_i(\rho_m)=1\quad\text{ and }\quad
s_i(\rho_m)=\rho_m-(x_i-x_{i+1})$$
for $i=1,2,\dots, m-1$.

\begin{corollary}Suppose $f\neq0$ is a pure meromorphic function on 
$(\C^*)^{m+r+2}$ of the type $\type(f)$. Let 
$\mu=\partial_i(\type(f))-h$,  then $\Cc^\mu_i(f)$ if pure of the type 
$$\type(\Cc_i^\mu(f))=s_i\big(\type(f)-\rho_m h\big)+\rho_m h\,.$$
\end{corollary}

\begin{proof}We have
\begin{align*}
s_i\big(\type(f)-\rho_m h\big)+\rho_m h&=s_i\big(\type(f)\big)-s_i(\rho_m)h+\rho_m h \\
&=s_i\big(\type(f)\big)-(\rho_m-(x_i-x_{i+1}))h+\rho_m h\\
&=s_i\big(\type(f)\big)+(x_i-x_{i+1})h\,,\end{align*}
which is equal to  \eqref{typin}.
\end{proof}

This means that admissible $\Cc_i^\mu$ operations act on types as the permutation of coordinates, provided that we shift the origin to $\rho_m h$.
\medskip

The operations $\Cbar_i^\mu$ differ from $\Cc_i^\mu$ by the factor $\frac1{\delta(h,\mu)}$, therefore the resulting type differs by the summand $-h\mu$.
The coefficient $\mu$ is given by  Lemma \ref{wzor_na_nu}.
\medskip

It is convenient to define the function $\phi_w:\C^m\to \C$ for each permutation $w\in \Spe_m$. As before let $x_i$ denote the standard coordinates of $\C^m$ on which $\Spe_m$ acts permuting the indices 
(action by $w^{-1}$ on indices).  We define
\begin{equation}\label{phi}\phi_w\left(\sum_{i=1}^m \alpha_i x_i\right)=
\sum_{i<j,\;w(i)>w(j)} \alpha_i-\alpha_j
\,.\end{equation}
The functions $\phi_w$ satisfy the cocycle condition
\begin{equation}\label{cocycl}\phi_{wv}(\alpha)=\phi_w(v(\alpha))+\phi_v(\alpha)\,.\end{equation}
We can evaluate the function $\phi_w$ on quadratic forms which are types of $f\in \AA_{m,r}$ since these types are  linear in $x_i$, hence formally we extend $\phi_w$ linearly, allowing $a_i$ to belong to the  vector space spanned by the parameters $\mub$. 
We summarize the consideration of types by the following theorem:

\begin{theorem}\label{wyznaczenienu} 
Let 
$\w=s_{i_1}s_{i_2}\dots s_{i_\ell}$ 
be a presentation of a permutation $w\in\Spe_m$, not necessarily reduced. Suppose $f$ is a pure meromorphic function. 
Denote by $\Cbar_\w^\ulu(f)$ the composition 
$$\Cbar_{i_1}^{\nu_1}\Cbar_{i_2}^{\nu_2}\dots \Cbar_{i_\ell}^{\nu_\ell}(f)\,,$$
with the coefficients $\nu_k$ chosen so that the operations preserve purity
$$\nu_k=\partial_{i_k}\type\big(\Cbar_{i_{k+1}}^{\nu_{k+1}}\Cbar_{i_{k+2}}^{\nu_{k+2}}\dots \Cbar_{i_\ell}^{\nu_\ell}(f)\big)-h\,.$$
Assume that $\nu_k+h\not\equiv 0 \mod 2\pi \ii\langle 1,\tau\rangle$ (or with the multiplicative notation $h\nu_k\not\equiv  q^{\mathbb Z}$), so that $\Cbar_i^\mu$ is defined.

Then the function $\Cbar_\w^\ulu(f)$ does not depend on the presentation of $w$. Moreover,
suppose 
$$\type(f)=q_x+\rho_m h+q_\mu\,,$$
where $q_x=\sum_{i=1}^m \alpha_i x_i$ and $q_\mu$ does not depend on the variables $x_i$, $i=1,2,\dots, m$.
Then
$$\type\big(\Cbar_\w^\ulu(f)\big)=
w(q_x)+\rho_m h+q_\mu-\phi_w(q_x)h\,.$$ 
\end{theorem}
\begin{proof}
First let us compute the type of $\Cbar_\w^\ulu(f)$. For $\ell=1$
$$\type\big(\Cbar^{\nu_1}_{k}(f)\big)=\rho_m h+s_{k}(q_x)- \nu_1 h\,,$$
where $$\nu_1=\partial_{k}\left(\rho_m h+q_x+q_\mu\right)-h=\alpha_{k}-\alpha_{k+1}=\phi_{s_{k}}(q_x)$$
by  Lemma \ref{wzor_na_nu}, as claimed. Further we argue by induction. If $\w=s_k\w'$ then 
$$\type(\Cbar_{\w'}^\ulu(f))=\rho_m h+{w'}(q_x)+q_\mu-\phi_{w'}(q_x)h$$
by the inductive assumption. Then
\begin{align*}\type(\Cbar_{\w}^\ulu(f))&=\rho_m h+s_k(w'(q_x))+q_\mu-\phi_{w'}(q_x)h-
\phi_{s_k}(w'(q_x))h\\&=\rho_m h+w(q_x)+q_\mu-\phi_w(q_x)h\end{align*}
by \eqref{cocycl}.

To show that $\Cbar_\w^\ulu(f)$ does not depend on the presentation of $\w$ let us check the braid and quadratic relation.
It is enough to consider the case $n=3$, $q_x=\alpha_1x_1+\alpha_2x_2+\alpha_3x_3$, $\alpha_i\in \C^{r+1}$.
Examining $\Cbar_1^{\nu_1}\Cbar_2^{\nu_2}\Cbar_1^{\nu_3}(f)$ we find that admissible values of $\nu_i$ are the following
$$\nu_1=\alpha_2-\alpha_3\,,\qquad \nu_2=\alpha_1-\alpha_3\,,\qquad\nu_3=\alpha_1-\alpha_2\,,$$
(written in the additive notation).
For the admissible operation $\Cbar_2^{\mu_1}\Cbar_1^{\mu_2}\Cbar_2^{\mu_3}(f)$ we have
$$\mu_1=\alpha_1-\alpha_2\,,\qquad \mu_2=\alpha_1-\alpha_3\,,\qquad\mu_3=\alpha_2-\alpha_3\,.$$
 The braid relation \eqref{redbraid} applies.
For the quadratic relation 
$$\Cbar_1^{\nu_1}\Cbar_1^{\nu_2}(f)=f$$
we note that if $q_x=\alpha_1x_1+\alpha_2x_2$, then $\nu_2=\alpha_1-\alpha_2$. Since $s_1(\alpha_1x_1+\alpha_2x_2)=\alpha_2x_1+\alpha_1x_2$ then $\nu_1=\alpha_2-\alpha_1=-\nu_2$.
The quadratic relation \eqref{redkwadrat} applies. Hence, $\Cbar_i^\ulu\Cbar_i^\ulu(f)=f$.
\end{proof}

\section{Elliptic classes of link patterns}\label{section-ellclasses}
We construct elliptic classes of link patterns inductively applying the action of $\Cc_i^\ulu$. The starting point is the Borisov-Libgober class
$$\Ep(\pa^{\min}_{m,r}):=\Elclas(X^{\min}_{m,r},\Sigma_{i=1}^r\lambda_iD_i)$$
 given by \eqref{poczatek}.
 Suppose that a link pattern $\pa$ is obtained by applying a permutation $w\in \Spe_m$ to $\pa^{\min}_{m,r}$. We assume that $w$ has a minimal length  among $w\in \Spe_m$ such that  $\pa=w \pa^{\min}_{m,r}$.
 We define \begin{equation}\label{eppadef}\Ep(\pa)~=~\Cc^\ulu_w(\Ep(\pa^{\min}_{m,r}))\,.\end{equation} To show that the definition does not depend on the choice of the permutation $w$ we analyze the elementary transformation of the elliptic class given by $\Cc^\ulu_i$. The first step is to trace how the type of the elliptic class is affected.

 \begin{example}\rm We keep the notation of \S\ref{beginning}: 
 $\mu_i=h^{1-\lambda_i}$.
The elliptic class of the link pattern 
 $\pa^{\min}_{8,2}$
is equal to
$$\delta\big(u\tfrac{x_1}{x_7},\mu_1\big)\delta\big(u\tfrac{x_2}{x_8},\mu_2\big)\delta\big(u\tfrac{x_1}{x_8},h\big).$$
It has the type equal to
$$(u+x_1-x_7)\mu_1+(u+x_2-x_8)\mu_2+(u+x_1-x_8)h.$$\end{example}
In general, \begin{equation*}\type(\Ep(\pa^{\min}_{m,r}))=\sum _{i=1}^r\left((u+x_i-x_{m-r+i})\mu_i+\sum_{j=m-r+i+1}^m (u+x_i-x_j)h\right)
\end{equation*}
The coefficient of $x_i$ in $$\type(\Ep(\pa^{\min}_{m,r})) -h\rho_m
$$
is equal to
\begin{equation}\label{wpolczynnikixi}\vvv(i)=\begin{cases}
rh+\mu_i& \text{ for }1\leq i\leq r,\\
ih& \text{ for }r+1\leq i\leq m-r,\\
(m-r+1)h-\mu_{i-m+r}& \text{ for }m-r+1\leq i\leq m\,.\end{cases}\end{equation}
We label the nodes of the link pattern with values $\vvv(i)$ written multiplicatively. In the example  above
\begin{equation}\label{przykladnumerowania}
\xymatrix@-1.4pc{
\mu_1h^2\ar@{..}[r]\ar@{<-}@/^1.8pc/^{\mu_1\phantom{mmm}}[rrrrrr]
&
\mu_2h^2\ar@{..}[r]\ar@{<-}@/^1.8pc/^{\phantom{mmmm}\mu_2}[rrrrrr]
&
h^3\ar@{..}[r]
&	
h^4\ar@{..}[r]
&
h^5\ar@{..}[r]
&
h^6\ar@{..}[r]
&
\tfrac{h^7}{\mu_1}\ar@{..}[r]
&\tfrac{h^7}{\mu_2}\,.
}
\end{equation}
Action of the permutation rearranges the values $\vvv(i)$. For example
applying the permutation
$$w:\quad 1\mapsto 3,\quad 2\mapsto 6,\quad 3\mapsto 1,\quad 4\mapsto 2,\quad 5
\mapsto 5,\quad 6\mapsto 8,\quad 7\mapsto 4,\quad 8\mapsto 7$$
to \eqref{przykladnumerowania} we obtain the link pattern
\begin{equation}\label{przykladnumerowania2}
\xymatrix@-1.4pc{
&
h^3\ar@{..}[r]
&
h^4\ar@{..}[r]
&
\mu_1h^2\ar@{..}[r]\ar@{<-}@/^2.4pc/^{\mu_1}[rrrr]
&
\tfrac{h^7}{\mu_2}\ar@{..}[r]
\ar@{->}@/^2pc/_{\mu_2}[rr]
&
h^5\ar@{..}[r]
&
\mu_2h^2\ar@{..}[r]
&
\tfrac{h^7}{\mu_1}\ar@{..}[r]
&h^6\,.
}
\end{equation} 
From this presentation it is easy to read admissible $\nu$ of the operation $\Cc^\nu_i$:
if $\pa=w\cdot \pa^{\min}_{m,r}$, namely 
$$\nu=\vvv(w(i))-\vvv(w(i+1))\qquad\text{(written additively)}\,,$$
by Lemma \ref{wzor_na_nu}.
Note that if we transpose consecutive loose nodes then the operation $\Cbar^\ulu_i$ is not defined. Indeed, $\nu=h^k/h^{k+1}=h^{-1}$  
and the normalizing factor \eqref{normalizacja} would be $\delta(h^{-1},h)=\frac{\vt(1)}{\vt(h^{-1})\vt( h)}$, hence equal to zero.

\begin{example}\rm Let us apply the operation $\Cc_3^\ulu$ to $\Ep(\pa_{8,2}^{\min})$. We read weights from \eqref{przykladnumerowania}.
Since $\Ep(\pa_{8,2}^{\min})=s_3(\Ep(\pa_{8,2}^{\min}))$ and $\nu=h^{-1}$ we have
$$\Cc_3^\ulu(\Ep(\pa_{8,2}^{\min}))=\delta\big(\tfrac{x_4}{x_3},h^{-1})\Ep(\pa_{8,2}^{\min})+\delta\big(\tfrac{x_3}{x_4},h)\Ep(\pa_{8,2}^{\min})=0\,.$$
Similarly, we apply $s_1$ to the pattern \eqref{przykladnumerowania2}. Since $s_1w=ws_3$ we have $\Cc_1^\ulu(\Ep(w\pa^{\min}_{m,r}))=
\Cc_w^\ulu\Cc^\ulu_3(\Ep(\pa^{\min}_{m,r}))=0\,.$
\end{example}

\section{Six moves increasing orbit dimension}
We present a simple way to determine the parameter $\nu$ for which $\Cc^\nu_i$ acting on $\Cc^\ulu_w(\pa^{\min}_{r,m})$ is admissible. 
On the pictures there are presented relevant elementary transformations of link patterns. The parameters $\alpha$ and $\beta$  attached to  
arcs belong to $(\C^*)^{r}$. After applying elementary transposition $s_i$ the configuration changes from $\pa$ to $s_i\pa$, but the parameters of arcs remain the same. At the right there is given the admissible value of parameter $\nu$ such that $\Cc_i^\nu(\Ep(\pa))$ is pure, and hence $\Cc_i^\nu(\Ep(\pa))=\Ep(s_i\pa)$.

{\parindent0pt
\def\duzyskip{\vskip30pt}
$
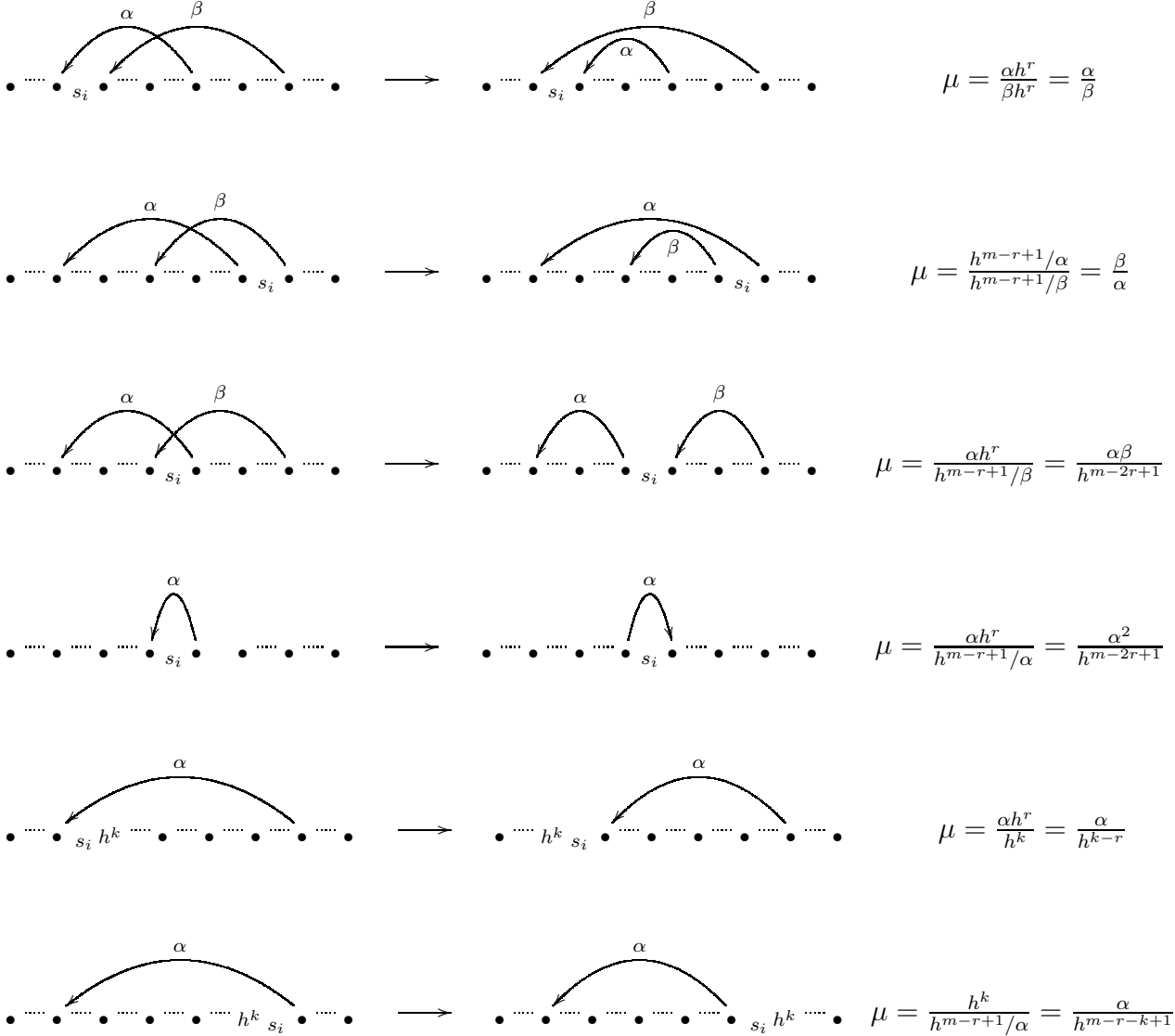
\xymatrix@-1.4pc{
_\bullet\ar@{..}[r]
&
_\bullet\ar@ {~}[r]_{s_i}
\ar@{<-}@/^1.8pc/^\alpha[rrr]
&
_\bullet\ar@{..}[r]
\ar@{<-}@/^1.8pc/^\beta[rrrr]&
_\bullet\ar@{..}[r]
&
_\bullet\ar@{..}[r]&
_\bullet\ar@{..}[r]
&
_\bullet\ar@{..}[r]
&_\bullet
&\ar@{->}[rr]
&&&
_\bullet\ar@{..}[r]
&
_\bullet\ar@ {~}[r]_{s_i}
\ar@{<-}@/^1.8pc/^\beta[rrrrr]
&
_\bullet\ar@{..}[r]
\ar@{<-}@/^1.4pc/_\alpha[rr]
&
_\bullet\ar@{..}[r]
&
_\bullet\ar@{..}[r]&
_\bullet\ar@{..}[r]
&
_\bullet\ar@{..}[r]
&_\bullet
}
$
\hfil $\nu=\frac {\alpha h^r}{\beta h^r}=\frac \alpha\beta$
\duzyskip

$
\xymatrix@-1.4pc{
_\bullet\ar@{..}[r]
&
_\bullet\ar@{..}[r]
\ar@{<-}@/^1.8pc/^\alpha[rrrr]
&
_\bullet\ar@{..}[r]
&
_\bullet\ar@{..}[r]
\ar@{<-}@/^1.8pc/^\beta[rrr]&
_\bullet\ar@{..}[r]&
_\bullet\ar@ {~}[r]_{s_i}
&
_\bullet\ar@{..}[r]
&_\bullet
&\ar@{->}[rr]
&&&
_\bullet\ar@{..}[r]
&
_\bullet\ar@{..}[r]
\ar@{<-}@/^1.8pc/^\alpha[rrrrr]
&
_\bullet\ar@{..}[r]
&
_\bullet\ar@{..}[r]
\ar@{<-}@/^1.4pc/_\beta[rr]&
_\bullet\ar@{..}[r]&
_\bullet\ar@ {~}[r]_{s_i}
&
_\bullet\ar@{..}[r]
&_\bullet
}
$
\hfil $\nu=\frac{h^{m-r+1}/ \alpha}{h^{m-r+1}/ \beta}=\frac\beta\alpha$

\duzyskip
$
\xymatrix@-1.4pc{
_\bullet\ar@{..}[r]
&
_\bullet\ar@{..}[r]
\ar@{<-}@/^1.8pc/^\alpha[rrr] 
&
_\bullet\ar@{..}[r]
&
_\bullet\ar@ {~}[r]_{s_i}
\ar@{<-}@/^1.8pc/^\beta[rrr]&
_\bullet\ar@{..}[r]&
_\bullet\ar@{..}[r]
&
_\bullet\ar@{..}[r]
&_\bullet
&\ar@{->}[rr]
&&&
_\bullet\ar@{..}[r]
&
_\bullet\ar@{..}[r]
\ar@{<-}@/^1.8pc/^\alpha[rr]
&
_\bullet\ar@{..}[r]
&
_\bullet\ar@ {~}[r]_{s_i}
&
_\bullet\ar@{..}[r]\ar@{<-}@/^1.8pc/^\beta[rr]
&
_\bullet\ar@{..}[r]
&
_\bullet\ar@{..}[r]
&_\bullet
}
$
\hfil $\nu=\frac {\alpha h^r}{h^{m-r+1}/\beta}=\frac {\alpha\beta}{h^{m-2r+1}}$

\duzyskip
$
\xymatrix@-1.4pc{
_\bullet\ar@{..}[r]
&
_\bullet\ar@{..}[r]
&
_\bullet\ar@{..}[r]
&
_\bullet\ar@ {~}[r]_{s_i}
\ar@{<-}@/^1.8pc/^\alpha[r]
&
_\bullet
&
_\bullet\ar@{..}[r]
&
_\bullet\ar@{..}[r]
&_\bullet
&\ar@{->}[rr]
&&&
_\bullet\ar@{..}[r]
&
_\bullet\ar@{..}[r]
&
_\bullet\ar@{..}[r]
&
_\bullet\ar@ {~}[r]_{s_i}\ar@{->}@/^1.8pc/^\alpha[r]
&
_\bullet\ar@{..}[r]
&
_\bullet\ar@{..}[r]
&
_\bullet\ar@{..}[r]
&_\bullet
}
$
\hfil $\nu=\frac {\alpha h^r}{h^{m-r+1}/\alpha}=\frac {\alpha^2}{h^{m-2r+1}}$

\duzyskip
$
\xymatrix@-1.4pc{
_\bullet\ar@{..}[r]
&
_\bullet\ar@{~}[r]_{s_i}
\ar@{<-}@/^1.8pc/^\alpha[rrrrr]
&
_{h^k}\ar@{..}[r]
&
_\bullet\ar@{..}[r]
&
_\bullet\ar@{..}[r]&
_\bullet\ar@{..}[r]
&
_\bullet\ar@{..}[r]
&_\bullet
&\ar@{->}[rr]
&&&
_\bullet\ar@{..}[r]
&
_{h^k}\ar@ {~}[r]_{s_i}
&
_\bullet\ar@{..}[r]
\ar@{<-}@/^1.8pc/^\alpha[rrrr]
&
_\bullet\ar@{..}[r]
&
_\bullet\ar@{..}[r]&
_\bullet\ar@{..}[r]
&
_\bullet\ar@{..}[r]
&_\bullet}
$
\hfil $\nu=\frac {\alpha h^r}{h^{k}}=\frac \alpha{h^{k-r}}$
\duzyskip
$
\xymatrix@-1.4pc{
_\bullet\ar@{..}[r]
&
_\bullet\ar@{..}[r]
\ar@{<-}@/^1.8pc/^\alpha[rrrrr]
&
_\bullet\ar@{..}[r]
&
_\bullet\ar@{..}[r]
&
_\bullet\ar@{..}[r]&
_{h^k}\ar@ {~}[r]_{s_i}
&
_\bullet\ar@{..}[r]
&_\bullet
&\ar@{->}[rr]
&&&
_\bullet\ar@{..}[r]
&
_\bullet\ar@{..}[r]
\ar@{<-}@/^1.8pc/^{\alpha}[rrrr]
&
_\bullet\ar@{..}[r]
&
_\bullet\ar@{..}[r]
&
_\bullet\ar@{..}[r]&
_\bullet\ar@ {~}[r]_{s_i}
&
_{h^k}\ar@{..}[r]
&_\bullet
}
$
\hfil $\nu=\frac {h^k}{h^{m-r+1}/\alpha }=\frac {\alpha}{h^{m-r-k+1}}$
}
\medskip

The transformations with linking patterns having reversed one or more arrows obey similar rules.
Note that starting from the link pattern $\pa^{\min}_{m,r}$ we never have $\nu=1$, which would made the operation $\Cc^\nu_i$ impossible. 

\begin{corollary}Let 
$f=\Ep(\pa^{\min}_{m,r})$.
Suppose $s_{i_1}s_{i_2}\dots s_{i_\ell}$ is a reduced decomposition of $w\in \Spe_m$.
Then  the composition
$$\Cc_{i_1}^{\nu_1}\Cc_{i_2}^{\nu_2}\dots \Cc_{i_\ell}^{\nu_\ell}(f)$$
is well-defined and does not depend on the word decomposition.
Morover if
$$\type(f)=q_x+\rho_m h+q_\mu\,,$$ then
$$\type\big(\Cc_w^\ulu(f)\big)=
w(q_x)+\rho_m h+q_\mu\,.$$ 
\end{corollary}
\begin{proof} We note that for any two reduced decompositions of $w$ one can pass from one to another applying braid relations, 
\cite[Theorem 3.3.1]{BjBr}. \end{proof}

If $w$ preserves the order of loose nodes, then transposition of strings labelled by a power of $h$ never appears.
Hence, the admissible parameters $\nu_i$ are never equal to $h^{-1}$. Hence,  the reduced operations   $\Cbar^{\nu_i}_i$ are defined. 
Moreover, one does not have to assume that the word representing $w$ is reduced, but only that the strings coming from the loose nodes in the course of applications of $s_{i_j}$ do not cross.

\begin{corollary}
Assume that $w$ preserves the order of loose nodes.
Then  $\Cbar^\ulu_w(\Ep(\pa^{\min}_{m,r}))$ is well-defined and it
depends only on $w$.
\end{corollary}

\section{Independence of the result of $\Cbar^\ulu_w$ presentation}\label{niezaleznosc}

We prove independence of the elliptic class from the presentation of the link pattern. First we analyze the reduced operations $\Cbar^\ulu_i$.
\begin{theorem} \label{niezalezy} 
Let  $\pa$ be a  labelled link pattern of rank $r$ and  let $w\in \Spe_m$, $\sigma\in \Spe_r$ be  permutations such that $\pa=\sigma^\mu w\pa^{\min}_{m,r}$.
We assume that $w$ preserves the order of the loose nodes.
Then  $\Cbar^\ulu_w(\Ep(\pa^{\min}_{m,r}))$ is defined 
and the result
$$\sigma^\mu\Cbar_w^\ulu(\Ep(\pa^{\min}_{m,r}))$$
does not depend on the choice of $\sigma$ and $w$.
\end{theorem}

\begin{proof}
The product of groups $\Spe_m\times\Spe_r$ acts on the set of labelled link patterns. Suppose $\sigma_1^\mu w_1(\pa^{\min}_{m,r})=\sigma_2^\mu w_2(\pa^{\min}_{m,r})$ and both $w_1,w_2$ preserve the order of loose nodes. Then $w_1^{-1}w_2$ preserves the order of the loose nodes and
$(\sigma_1^\mu)^{-1}\sigma_2^\mu w_1^{-1} w_2(\pa^{\min}_{m,r})=\pa^{\min}_{m,r}$.
Hence,
it is enough to show that if $(\sigma,w)$ stabilizes the labelled link pattern $\pa^{\min}_{m,r}$ and $w$ preserves the order of the loose nodes (so  it is constant on the loose nodes), then $\sigma^\mu \Cbar^\ulu_w(\Ep(\pa^{\min}_{m,r}))=\Ep(\pa^{\min}_{m,r})$. 

The stabilizer of  $\pa^{\min}_{m,r}$ and loose nodes is generated by the simultaneous transpositions of arcs and labels:
$$\sigma=s_i\,,\quad w=s_is_{m-r+i}\,.$$
The calculus involves only the variables $x_i,x_{i+1},x_{m-r+i},x_{m-r+i+1},\mu_i,\mu_{i+1}$ it is enough to check the result of the action for $m=4$,$r=2$, $i=1$, as in \eqref{przyklad42} since:
$$\pa^{\min}_{m,r}={\mathcal M}\cdot\big((\pa^{\min}_{4,2})_{x_1:=x_i,\,x_2:=x_{i+1},\,x_3:=x_{m-r+i},\,x_4:=x_{m-r+i+1},\,\mu_1:=\mu_i,\,\mu_2:=\mu_{i+1}}
\big)$$
and ${\mathcal M}$ is a  symmetric function with respect to $x_i\leftrightarrow x_{i+1},\,x_{m-r+i}\leftrightarrow x_{m-r+i+1}$, does not depend on $\mu_i$, $\mu_{i+1}$.
The computation is carried out via a direct verification based on geometric considerations in Section \ref{basicexample}.
\end{proof}

\begin{remark}\rm Note that the set of permutations preserving the order of the loose nodes is not a group, hence in the argument above we had to analyze the composition $w_1^{-1} w_2$ to have a permutation fixing loose nodes.\end{remark}

\section{The basic example -- the flip relation}
\label{basicexample}
The purpose of this section is mainly to give a geometric proof of the four-term relation \cite[eq. (2.7)]{RTV} which plays a role in the proof of Theorem \ref{niezalezy}  for the elliptic algebra.
\medskip

Consider the minimal link pattern $\pa^{\min}_{4,2}$ and the corresponding $\B_4$-orbit in $\Hom(\C^4,\C^4)$
$$(X_{4,2}^{\min})^ \circ=B_4\cdot\hbox{\small$\begin{pmatrix}0&0&1&0\\0&0&0&1\\0&0&0&0\\0&0&0&0\end{pmatrix}$}=\left\{\hbox{\small$\begin{pmatrix}
0&0&a&b\\0&0&0&c\\0&0&0&0\\0&0&0&0\end{pmatrix}$}\;:\;a\neq 0\,,~~ c\neq 0\right\}\,.$$
The closure $X_{4,2}^{\min}$ is isomorphic to $\C^3$.
The effect of the actions of $s_1$ and $s_3$ are  equal to the $\B_4$ orbit 
$$X^\circ:=\B_4s_1
\cdot (X_{4,2}^{\min})^\circ
=\B_4s_3
\cdot (X_{4,2}^{\min})^\circ=\left\{\hbox{\small$\begin{pmatrix}
0&0&s&t\\0&0&u&v\\0&0&0&0\\0&0&0&0\end{pmatrix}$}\;:\;u\neq 0\,,~~ sv-tu\neq 0\right\}\,.$$ The closure $X=\overline{X^\circ}$ is isomorphic to $\C^4$.
Let $P_i\subset \GL_4$ be the minimal parabolic subgroup generated by $\B_4$ and $s_i$. We have two Bott-Samelson type resolutions, of the pair $(X,\partial X)$: 
$$Z_1=P_1\times_{\B_4}X~~\text{ and }~~Z_3=P_3\times_{\B_4}X\,.$$
Let us analyze  the first one. There are two charts: 
\begin{equation}\label{par1}\hbox{\small$\begin{pmatrix}
1&0&0&0\\z&1&0&0\\0&0&1&0\\0&0&0&1\end{pmatrix}$}\cdot\hbox{\small$\begin{pmatrix}
0&0&a&b\\0&0&0&c\\0&0&0&0\\0&0&0&0\end{pmatrix}$}=\hbox{\small$\begin{pmatrix}
0&0&a&b\\0&0&az&bz+c\\0&0&0&0\\0&0&0&0\end{pmatrix}$}\end{equation}
and
\begin{equation}\label{par2}\hbox{\small$\begin{pmatrix}
z&1&0&0\\1&0&0&0\\0&0&1&0\\0&0&0&1\end{pmatrix}$}\cdot\hbox{\small$\begin{pmatrix}
0&0&a&b\\0&0&0&c\\0&0&0&0\\0&0&0&0\end{pmatrix}$}=\hbox{\small$\begin{pmatrix}
0&0&az&c+bz\\0&0&a&b\\0&0&0&0\\0&0&0&0\end{pmatrix}$}\,.\end{equation}
Consider the following divisor on $X$
$$D=(1-\alpha) D_1+(1-\beta )D_2\,,\quad\text{ where }D_1=\{u=0\} \text{ and } D_2=\{sv-tu=0\}\,.$$
We explain  the meaning of the computation below. According to the definition of the elliptic class of a singular variety (see \cite{BoLi}, \cite[\S2.3]{RiWe1}), we choose a resolution of the singular pair
$$f:(Z,\widetilde D)\longrightarrow (X,D).$$
In our case, this is a small resolution of the three-dimensional quadratic cone. The resolution is $\T$-equivariant. The parametrizations of neighbourhoods of the fixed points in $Z$ are given by \eqref{par1} and \eqref{par2}. In these coordinates, the inverse images of the boundary divisors have simple normal crossings; they are given by the vanishing of coordinates. The multiplicities are obtained by applying pull-back intertwined with Grothendieck duality:
$$\widetilde D=f^*(K_X+D)-K_Z\,.$$
This calculation is easily carried out by interpreting a section of $K_X+D$ as a differential form with the prescribed poles along $D$.
Finally, to compute the restriction of the elliptic class at the fixed point corresponding to $X=\Hom(\C^m,\C^m)$, we apply the localization theorem; that is, we sum the contributions of the fixed points in $Y$.
\medskip

Here are the details:
we pull-back of $D$ via $f:Z_1\to X$ and check that it is a normal crossing divisor:
$$\widetilde D=f^*(K_X+D)-K_{Z_1}=(1-\alpha)\{az=0\}+(1-\beta) \{ac=0\}\,.$$
Indeed, in the first chart we have
\begin{multline*}f^*\left(u^{\alpha-1}(sv-tu)^{\beta-1}ds\wedge dt\wedge du\wedge dv\right)=(az)^{\alpha-1}(ac)^{\beta-1}
(a\,dz\wedge da\wedge db\wedge  dc)\\=z^{\alpha-1}a^{\alpha+\beta-1}c^{\beta-1} dz\wedge da\wedge db\wedge dc\,.
\end{multline*}
In the second chart 
\begin{multline*}
f^*\left(u^{\alpha-1}(sv-tu)^{\beta-1)}ds\wedge dt\wedge du\wedge dv\right)
=a^{\alpha-1} (-ac)^{\beta-1}(a \,dz\wedge da\wedge db\wedge  dc)\\\sim
a^{\alpha+\beta-1}  c^{\beta-1}dz\wedge da\wedge db\wedge  dc\,.\end{multline*}

We compute the localized equivariant Borisov-Libgober class 
using localization formula for the torus action. We skip the computation of the torus weights, and we refer to the almost identical computations for fundamental, CSM and motivic Chern classes \cite[Th.~4.1, Th.~5.1]{RuWe}. 
The result is
\begin{equation}\label{rownanie42}\Elclas(X,D)=\delta\big(\tfrac{x_2}{x_1},h^\alpha\big)\Elclas(X^{\rm min}_{4,2},D_0)+\delta\big(\tfrac{x_1}{x_2},h\big)s_1\Elclas(X^{\rm min}_{4,2},D_0)\,,\end{equation}
where 
$$D_0=(1-\alpha-\beta)\{a=0\}+(1-\beta)\{c=0\}\,.$$
The first summand of \eqref{rownanie42} is equal to the contribution of the fixed point from the first chart (there the component $\{z=0\}$ has the coefficient $1-\alpha$). In the second chart the component 
$\{z=0\}$ has coefficient 0, it does not enter into the boundary divisor. Hence the factor $\delta\big(\tfrac{x_1}{x_2},h\big)$ in the second summand.
The multiplicities   of at the boundary divisor are related to the $\mu_i$ variables by the formula \eqref{muintro} hence
 \begin{equation}\label{pierwsze}\mu_1=h^{1-(1-\alpha-\beta)}=h^{\alpha+\beta}\,,\qquad  \mu_2=h^{1-(1-\beta)}=h^{\beta}\end{equation}
and  the admissible parameter is equal to $$\frac{\mu_1}{\mu_2}=h^\alpha\,.$$
In terms of the Demazure operators \eqref{rownanie42} reads

$$\Elclas(X,D)=\Cc_1^{\mu_1/\mu_2}\big(
\Elclas(X^{\rm min}_{4,2},D_0)\big)\,.
$$
An alternative resolution is obtained by applying the action of $s_3$
$$\hbox{\small$\begin{pmatrix}
1&0&0&0\\0&1&0&0\\0&0&1&0\\0&0&z&1\end{pmatrix}$}\cdot\hbox{\small$\begin{pmatrix}
0&0&a&b\\0&0&0&c\\0&0&0&0\\0&0&0&0\end{pmatrix}$}=\hbox{\small$\begin{pmatrix}
0&0&a-bz&b\\0&0&-cz&c\\0&0&0&0\\0&0&0&0\end{pmatrix}$}$$
(we recall, that we act by conjugation). Then
$$f^*\left(u^{\alpha-1}(sv-tu)^{\beta-1}ds\wedge dt\wedge du\wedge dv\right)\sim (cz)^{\alpha-1}(ac)^{\beta-1}
c\,da\wedge db\wedge dz\wedge dc$$ $$=z^{\alpha-1}a^{\beta-1}c^{\alpha+\beta-1}da\wedge db\wedge dz\wedge dc\,.
$$
Thus,
$$\Elclas(X,D)=\Cc_3^{\alpha}\big(
\Elclas(X^{\rm min}_{4,2},D_0')\big)\,.
$$
with
$$D_0'=(1-\alpha)\{a=0\}+(1-\alpha-\beta)\{c=0\}\,.$$

\begin{equation}\label{drugie}\mu_1'=h^\beta\,,\qquad \mu_2'=h^{\alpha+\beta}\end{equation}
we obtain
$$\Elclas(X,D)=\Cc_3^{\mu'_2/\mu'_1}\big(
\Elclas(X,D_0')\big)\,.
$$
Note that the substitutions \eqref{pierwsze} and \eqref{drugie} differ by the transposition $\mu_1\leftrightarrow\mu_2'$, $\mu_2\leftrightarrow\mu_1'$. Since the elliptic class 
does not depend on a resolution of singularities we obtain:
\begin{corollary}We have the identity 
\begin{equation}\label{flipcomp}\Cc_1^{\mu_1/\mu_2}(\mathcal A)=
s_1^\mu\Cc_3^{\mu_2/\mu_1}(\mathcal A)\,,\end{equation}
where
$$\mathcal A=
\delta\big(\tfrac{x_1}{x_3},\mu_1\big)
\delta\big(\tfrac{x_2}{x_4},\mu_2\big)
\delta\big(\tfrac{x_1}{x_4},h\big)\,.$$
Analogous identity holds for the reduced operations $\Cbar^\nu_i$.
\end{corollary}

We rewrite the identity \eqref{flipcomp} using the definition of the operation $\Cc_i^\mu$:
\begin{multline}\label{flipkonkretnie}
\delta
   \big(\tfrac{x_1}{x_4},h\big)
   \delta
   \big(\tfrac{x_2}{x_1},\tfrac{\mu_1}{\mu_2}\big) \delta
   \big(\tfrac{x_1}{x_3},\mu_1\big) \delta
   \big(\tfrac{x_2}{x_4},\mu_2\big)
   +
\delta \big(\tfrac{x_1}{x_2},h\big)
   \delta
   \big(\tfrac{x_2}{x_4},h\big)
   \delta \big(\tfrac{x_2}{x_3},\mu_1\big) \delta
   \big(\tfrac{x_1}{x_4},\mu_2\big)
   =\\
   =\delta
   \big(\tfrac{x_1}{x_4},h\big)
   \delta \big(\tfrac{x_1}{x_3},\mu_2\big) \delta
   \big(\tfrac{x_2}{x_4},\mu_1\big) \delta
   \big(\tfrac{x_4}{x_3},\tfrac{\mu_1}{\mu_2}\big)+
   \delta \big(\tfrac{x_1}{x_3},h\big)
   \delta
   \big(\tfrac{x_3}{x_4},h\big)
   \delta \big(\tfrac{x_2}{x_3},\mu_1\big) \delta
   \big(\tfrac{x_1}{x_4},\mu_2\big)
   \end{multline}
After applying the definition of $\delta$ and multiplying by the common denominator, simplifying by the relation $\vt\big(\frac ab\big)=-\vt\big(\frac ba\big)$, we obtain the following monstrous relation\footnote{We apply the multiplicative notation for the arguments of the theta function.}
\begin{multline*}
$$\vt\big(\tfrac{y_2}{y_1}\big) \Big(\vt(h) \vt\big(\tfrac{\mu_2
   x_2}{\mu_1 x_1}\big) \vt\big(\tfrac{x_2}{y_1}\big)
   \vt\big(\tfrac{h x_1}{y_2}\big) \vt\big(\tfrac{\mu_1
   x_2}{y_2}\big) \vt\big(\tfrac{\mu_2
   x_1}{y_1}\big)\\
   -\vt\big(\tfrac{\mu_2}{\mu_1}\big)
   \vt\big(\tfrac{h x_1}{x_2}\big) \vt\big(\tfrac{x_1}{y_1}\big)
   \vt\big(\tfrac{h x_2}{y_2}\big) \vt\big(\tfrac{\mu_1
   x_1}{y_2}\big) \vt\big(\tfrac{\mu_2 x_2}{y_1}\big)\Big)=
   \end{multline*}
    \begin{multline*}
   =
   \vt\big(\tfrac{x_2}{x_1}\big) \Big(\vt(h) \vt\big(\tfrac{x_2}{y_1}\big)
   \vt\big(\tfrac{\mu_2 y_2}{\mu_1 y_1}\big) \vt\big(\tfrac{h
   x_1}{y_2}\big) \vt\big(\tfrac{\mu_1 x_1}{y_1}\big)
   \vt\big(\tfrac{\mu_2 x_2}{y_2}\big)\\
   -\vt\big(\tfrac{\mu_2}{\mu_1}\big) \vt\big(\tfrac{h y_1}{y_2}\big)
   \vt\big(\tfrac{x_2}{y_2}\big) \vt\big(\tfrac{h x_1}{y_1}\big)
   \vt\big(\tfrac{\mu_1 x_1}{y_2}\big) \vt\big(\tfrac{\mu_2
   x_2}{y_1}\big)\Big)\,.\end{multline*}

The presented calculation are further continued in Example \ref{Schred} which shows how the elliptic class of a pattern specializes to the class of a Schubert variety --- in this case for $\PP^1$.

\section{Unnormalized  elliptic classes}
Before showing that the unnormalized elliptic classes of link patterns are well-defined let us analyse the first nontrivial example. We have to check that the normalizing factor $\prod \delta(\nu_i,h)$ does not change when we change the arrow labels, and accordingly the permutation of nodes.

\def\aao{{\alpha^{\rm t}}}
\def\bbo{{\beta^{\rm t}}}
\def\aai{{\alpha^{\rm s}}}
\def\bbi{{\beta^{\rm s}}}

\def\pataoboaibi{\xymatrix@-1.5pc{
\aao\ar@{..}[r]\ar@{<-}@/^2.5pc/[rr]&
\bbo\ar@{..}[r]\ar@{<-}@/^2.5pc/[rr]&
\aai\ar@{..}[r]&
\bbi}}
\def\pataibiaobo{\xymatrix@-1.5pc{
\aai\ar@{..}[r]\ar@{->}@/^2.5pc/[rr]&
\bbi\ar@{..}[r]\ar@{->}@/^2.5pc/[rr]&
\aao\ar@{..}[r]&
\bbo}}
\def\pataobiaibo{\xymatrix@-1.5pc{
\aao\ar@{..}[r]\ar@{<-}@/^2.5pc/[rr]&
\bbi\ar@{..}[r]\ar@{->}@/^2.5pc/[rr]&
\aai\ar@{..}[r]&
\bbo}}
\def\pataiboaobi{\xymatrix@-1.5pc{
\aai\ar@{..}[r]\ar@{->}@/^2.5pc/[rr]&
\bbo\ar@{..}[r]\ar@{<-}@/^2.5pc/[rr]&
\aao\ar@{..}[r]&
\bbi}}
\def\pataobobiai{\xymatrix@-1.5pc{
\aao\ar@{..}[r]\ar@{<-}@/^2.5pc/[rrr]&
\bbo\ar@{..}[r]\ar@{<-}@/^1.5pc/[r]&
\bbi\ar@{..}[r]&
\aai}}
\def\pataibiboao{\xymatrix@-1.5pc{
\aai\ar@{..}[r]\ar@{->}@/^2.5pc/[rrr]&
\bbi\ar@{..}[r]\ar@{->}@/^1.5pc/[r]&
\bbo\ar@{..}[r]&
\aao}}
\def\pataibobiao{\xymatrix@-1.5pc{
\aai\ar@{..}[r]\ar@{->}@/^2.5pc/[rrr]&
\bbo\ar@{..}[r]\ar@{<-}@/^1.5pc/[r]&
\bbi\ar@{..}[r]&
\aao}}
\def\pataobiboai{\xymatrix@-1.5pc{
\aao\ar@{..}[r]\ar@{<-}@/^2.5pc/[rrr]&
\bbi\ar@{..}[r]\ar@{->}@/^1.5pc/[r]&
\bbo\ar@{..}[r]&
\aai}}
\def\pataoaibobi{\xymatrix@-1.5pc{
\aao\ar@{..}[r]\ar@{<-}@/^2pc/[r]&
\aai\ar@{..}[r]&
\bbo\ar@{..}[r]\ar@{<-}@/^2pc/[r]&
\bbi}}
\def\pataiaobobi{\xymatrix@-1.5pc{
\aai\ar@{..}[r]\ar@{->}@/^2pc/[r]&
\aao\ar@{..}[r]&
\bbo\ar@{..}[r]\ar@{<-}@/^2pc/[r]&
\bbi}}
\def\pataoaibibo{\xymatrix@-1.5pc{
\aao\ar@{..}[r]\ar@{<-}@/^2pc/[r]&
\aai\ar@{..}[r]&
\bbi\ar@{..}[r]\ar@{->}@/^2pc/[r]&
\bbo}}
\def\pataiabibo{\xymatrix@-1.5pc{
\aai\ar@{..}[r]\ar@{->}@/^2pc/[r]&
\aao\ar@{..}[r]&
\bbi\ar@{..}[r]\ar@{->}@/^2pc/[r]&
\bbo}}


\def\ptaibiboao{\begin{matrix}\pataibiboao
 \\ 
 \left\{\frac{\alpha}{\beta},\frac{\alpha^2}{h},\frac{\alpha\beta}{h},\frac{\alpha\beta}{h},\frac{\beta^2}{h}\right\}
\end{matrix}}

\def\ptaobiaibo{\begin{matrix}\pataobiaibo
 \\ 
 \left\{\frac{\beta}{\alpha},\frac{\alpha\beta}{h},\frac{\beta^2}{h}\right\}
 \end{matrix}}

\def\ptaibiaobo{\begin{matrix}\pataibiaobo
 \\ 
 \left\{\frac{\alpha\beta}{h},\frac{\beta^2}{h},\frac{\alpha^2}{h},\frac{\alpha\beta}{h}\right\}
\end{matrix}}

\def\ptaobiboai{\begin{matrix}\pataobiboai
 \\ 
\left\{\frac{\beta^2}{h},\frac{\beta}{\alpha}\right\}
\end{matrix}}

\def\ptaibobiao{\begin{matrix}\pataibobiao
 \\ 
 \left\{\frac{\alpha\beta}{h},\frac{\alpha}{\beta},\frac{\alpha^2}{h},\frac{\alpha\beta}{h}\right\}
\end{matrix}}

\def\ptaoaibibo{\begin{matrix}\pataoaibibo
 \\ 
\left\{\frac{\alpha\beta}{h},\frac{\beta^2}{h}\right\}
\end{matrix}}

\def\ptaiboaobi{\begin{matrix}\pataiboaobi
 \\ 
 \left\{\frac{\alpha}{\beta},\frac{\alpha^2}{h},\frac{\alpha\beta}{h}\right\}
\end{matrix}}

\def\ptaoaibobi{\begin{matrix}\pataoaibobi
 \\ 
 \left\{\frac{\alpha\beta}{h}\right\}
\end{matrix}}

\def\ptaiabibo{\begin{matrix}\pataiabibo
 \\ 
 \left\{\frac{\beta^2}{h},\frac{\alpha^2}{h},\frac{\alpha\beta}{h}\right\}
\end{matrix}}

\def\ptaobobiai{\begin{matrix}\pataobobiai
 \\ 
 \left\{\frac{\beta}{\alpha}\right\}
\end{matrix}}

\def\ptaiaobobi{\begin{matrix}\pataiaobobi
 \\ 
\left\{\frac{\alpha^2}{h},\frac{\alpha\beta}{h}\right\}
\end{matrix}}

\def\ptaoboaibi{\begin{matrix}\pataoboaibi
 \\ 
\{\}
\end{matrix}}

\def\bo#1{\boxed{#1}}

{\setlength\arraycolsep{-1pt}
\noindent\hfil$\begin{matrix}\phantom{xz}\\[20pt]
&&\bo{\ptaibiboao}\\[20pt]
&\hfill \Cc_2^{\beta^2/h}\nearrow&&\nwarrow\hskip-5pt\nwarrow {s_1^\mu \Cc_1^{\beta/\alpha},\Cc_3^{\alpha/\beta}}\hfill\\[20pt]
&\bo{\ptaibobiao}&&\bo{\ptaibiaobo}\\[20pt]
& s_1^\mu\Cc_1^{\alpha\beta/h}\nearrow\hfill\nwarrow \Cc_3^{\alpha\beta/h}&&\hfill\nwarrow \Cc_2^{\alpha\beta/h}\\[20pt]
\bo{\ptaobiaibo}&&\bo{\ptaiboaobi}&&\bo{\ptaiabibo}\\[20pt]
\uparrow \Cc_3^{\alpha\beta/h}&  s_1^\mu\Cc_1^{\alpha\beta/h}\hskip5pt\nearrow\hskip10pt\nwarrow\hskip5pt \Cc_2^{\beta/\alpha}&&\Cc_1^{\alpha^2/h}\hskip5pt\nearrow\hskip10pt\nwarrow\hskip5pt \Cc_2^{\alpha/\beta}&\uparrow \Cc_3^{\beta^2/h}\\[20pt]
\bo{\ptaobiboai}&&\bo{\ptaoaibibo}&&\bo{\ptaiaobobi}\\[20pt]
&\nwarrow \Cc_2^{\beta^2/h}\hfill&&\nwarrow \Cc_3^{\beta^2/h}\hfill \Cc_1^{\alpha^2/h}\nearrow\\[20pt]
&\bo{\ptaobobiai}&&\bo{\ptaoaibobi}\\[20pt]
&\hfill{s_1^\mu \Cc_1^{\alpha/\beta},
 \Cc_3^{\beta/\alpha}}\nwarrow\hskip-5pt\nwarrow && \nearrow \Cc_2^{\alpha\beta/h}\hfill\\[20pt]
&&\bo{\ptaoboaibi}
\end{matrix}
$}
\begin{example}\rm\label{przyklad42e}
Let $m=4$, $r=2$. We present the lattice of orbits, see the diagram above.
The minimal orbit is at the bottom. The link patterns at the second row (counting from the bottom) represent the maximal upper-triangular orbits; the arrows in the patterns are directed left. 
The arrows are labeled by elementary transformations of elliptic classes  $\Cc_i^\nu$ or $s_1^\mu\Cc_i^\nu$,  $i=1,2,3$  with the admissible $\nu$.
The variables  $\alpha=\mu_1$ and $\beta=\mu_2$ are associated to arrows. Instead of indicating arc labels we list the coefficients at the nodes: $\aao$ and $\bbo$ at the targets of the arrows and $\aai$ and $\bbi$ at the sources
$$\aao=\alpha h^2\,,\quad\bbo=\beta h^2\,,\quad \aai=\tfrac{h^3}\alpha\,,\quad\bbi=\tfrac{h^3}\beta\,.$$
In the braces there are given the arguments $\nu_k$ of the normalizing factors $\prod_{k=1}^{\ell{w}}\delta(\nu_k,h)$.

\end{example}

Before giving a proof of independence of the normalizing factor from the link presentation we  argue that it is enough to consider link patterns of the rank $r$ and the size $m=2r$. If $m>2r$ then we extend the link pattern adding $m-2r$ nodes and arrows from the new nodes to loose nodes, preserving the order. 
We set $\mu_i'=\frac{\mu_i} {h^{m/2-r}}$. Then
$$\mu_i h^r=\mu_i' h^{m/2}\,,\qquad \frac{h^{m-r+1}}{\mu_i}=\frac{h^{m/2+1}}{\mu_i'}\qquad\text{for }1\leq i\leq r\,.$$
We choose the remaining variables $\mu'_j$ for $j=r+1,\dots,m-r$ so that
$$\mu_j' h^{m/2}=h^j\,.$$
Multiplying all the coefficients by $h^{m/2-r}$ we obtain the distribution of coefficients as for the link pattern $\pa^{\min}_{2(m-r),m-r}$. 
Such operation does not change the quotients of coefficients we have to determine.

\begin{example}\rm Extending the link pattern 
$\pa^{\min}_{m,r}$ to 
$\pa^{\min}_{2(m-r),m-r}$:

\hskip110pt
$
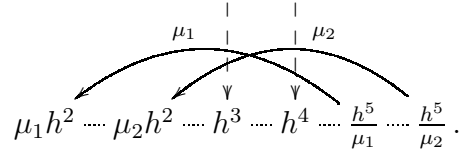
\xymatrix@-1.4pc{
&
&
&\ar@{-->}[ddd]
&\ar@{-->}[ddd]
\\ \\ \\
&
\mu_1h^2\ar@{..}[r]\ar@{<-}@/^2.4pc/^{\hskip-15pt\mu_1}[rrrr]
&
\mu_2h^2\ar@{..}[r]\ar@{<-}@/^2.4pc/^{\hskip25pt\mu_2}[rrrr]
&
h^3\ar@{..}[r]
&
h^4\ar@{..}[r]
&
\tfrac{h^5}{\mu_1}\ar@{..}[r]
&\tfrac{h^5}{\mu_2}\,.
}
$

\noindent Our procedure leads to 

\hskip100pt $
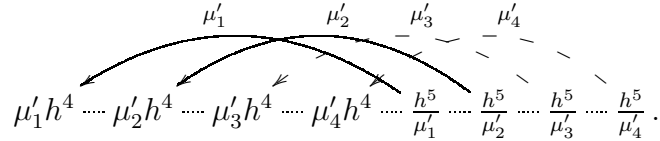
\xymatrix@-1.4pc{
&
\mu_1'h^4\ar@{..}[r]\ar@{<-}@/^2.4pc/^{\hskip-15pt\mu'_1}[rrrr]
&
\mu_2'h^4\ar@{..}[r]\ar@{<-}@/^2.4pc/^{\hskip15pt\mu'_2}[rrrr]
&
\mu_3'h^4\ar@{..}[r]\ar@{<--}@/^2.4pc/^{\hskip15pt\mu'_3}[rrrr]
&
\mu'_4h^4\ar@{..}[r]\ar@{<--}@/^2.4pc/^{\hskip15pt\mu'_4}[rrrr]
&
\tfrac{h^5}{\mu_1'}\ar@{..}[r]
&
\tfrac{h^5}{\mu_2'}\ar@{..}[r]
&
\tfrac{h^5}{\mu_3'}\ar@{..}[r]
&\tfrac{h^5}{\mu_4'}\,.
}$

\noindent Similarly we extend an arbitrary link pattern. For example

\hskip 110pt$
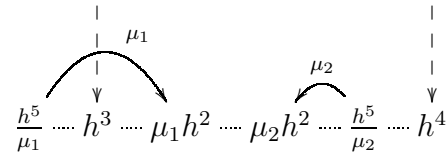
\xymatrix@-1.4pc{
&
&\ar@{-->}[ddd]
& &&
&\ar@{-->}[ddd]
\\ \\ \\
&
\tfrac{h^5}{\mu_1}\ar@{..}[r]\ar@{->}@/^2.4pc/^{\hskip25pt\mu_1}[rr]
&
h^3\ar@{..}[r]
&
\mu_1h^2\ar@{..}[r]
&
\mu_2h^2\ar@{..}[r]\ar@{<-}@/^1.4pc/^{\mu_2}[r]
&
\tfrac{h^5}{\mu_2}\ar@{..}[r]
&h^4
}
$

\noindent is extended to 

\hskip100pt$
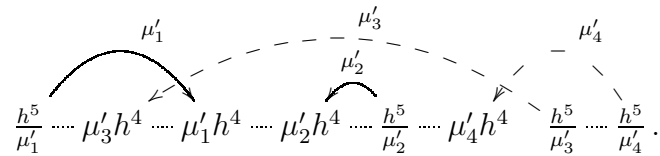
\xymatrix@-1.4pc{
&
\tfrac{h^5}{\mu'_1}\ar@{..}[r]\ar@{->}@/^2.4pc/^{\hskip25pt\mu_1'}[rr]
&
\mu_3'h^4\ar@{..}[r]\ar@{<--}@/^2.8pc/^{\hskip25pt\mu_3'}[rrrrr]
&
\mu_1'h^4\ar@{..}[r]
&
\mu_2'h^4\ar@{..}[r]\ar@{<-}@/^1.4pc/^{\mu_2'}[r]
&
\tfrac{h^5}{\mu'_2}\ar@{..}[r]
&\mu_4'h^4\ar@{<--}@/^2.4pc/^{\hskip25pt\mu_4'}[rr]
&\tfrac{h^5}{\mu_3'}\ar@{..}[r]
&\tfrac{h^5}{\mu_4'}
\,.
}
$\end{example}

Now we are at the position to prove:

\begin{theorem}\label{main}Let $\pa$ be a labelled link pattern of the size $m$ and rank $r$. The  function 
$\sigma^\mu\Cc_w^\ulu(\Ep(\pa^{\min}_{m,r}))$ does not depend on 
$(\sigma,w)\in \Spe_r\times\Spe_m$,
provided that $w$ has minimal length among all the possible pairs $(\sigma,w)$ satisfying \begin{equation}\label{przedstawienie}\pa=\sigma^\mu w\pa^{\min}_{m,r}\,.\end{equation}
\end{theorem}

\begin{proof}We can assume that $m$ is even and the rank $r=m/2$. Suppose the length of $w$ is equal to $\ell$ and let 
$w=s_{i_1}s_{i_2}\dots s_{i_\ell}$ be a reduced decomposition. 
The functions $\Cc_w^\ulu(\Ep(\pa^{\min}_{m,r}))$ and $\Cbar_w^\ulu(\Ep(\pa^{\min}_{m,r}))$ differ by the product $\prod_{i=k}^\ell \delta(\nu_k,h)$ 
where $\nu_k$ is the parameter of $k$-th operation  $\Cc^{\nu_k}_{i_k}$. The list of $\nu_k$ parameters can be read from the picture representing the link pattern. 
If the arrow with the label $\alpha$ is reversed then $\frac{\alpha^2}h$ appears. To determine which of coefficients $\frac\alpha\beta$ or   
$\frac{\alpha \beta}h$ appears, it is enough to analyze the relative position of the corresponding pair of arrows. The situation is reduced to the case $m=4$, $r=2$, where independence of the presentation \eqref{przedstawienie} is verified directly. The list of parameters is given in the table of Example \ref{przyklad42e}.\end{proof}

\section{Relation with the elliptic classes of Schubert varieties}
Our initial aim was to associate  elliptic classes to link patterns so that in the case $m=2n$, $r=n$. When a link pattern represents a permutation of $n$ elements, after normalization and restriction as in \cite[eq (2.34)]{RTV} (see \cite[\S6.1]{RiWe1})  we would recover elliptic classes of Schubert varieties.
The corresponding procedure division and restriction is simply related to division by  the Borel group $\B_n$, analogously as in the case of the twisted motivic Chern classes, \cite[\S12]{KoWe2} .
\medskip

Let us assume that all the arrows of the link patterns have targets at nodes with positions $i\leq n$. Such link patterns $\pa$ define a permutation $w_\pa$.
We rename the equivariant variables: we do not change the equivariant variables $x_i$ for $i\leq n$ and let $$x_{j+n}=y_j\qquad \text{for }j=1,2,\dots,n.$$
The additional variable $u$ is specialized to 1. 
We introduce the normalization
$$eu^{ell}_M=\prod_{i=1}^{n}\prod_{j=1}^{n}\vt\big(\tfrac {x_i}{y_j}\big)\,.$$
This is the equivariant elliptic Euler class of the space of $n\times n$-matrices with $(\C^*)^n\times (\C^*)^n$ action by left and right multiplication.
Let
$$eu^{ell}_{\Fl}=\prod_{n\geq i>j\geq 1}\vt\big(\tfrac {y_i}{y_j}\big)$$
be another normalizing factor, which after the application of the Kirwan map \eqref{Kirwan} becomes the equivariant elliptic Euler class of the tangent bundle to the flag variety.
Furthermore we need
$$\mathcal B=\prod_{1\leq i<j\leq n}\tfrac{\vt\left(\tfrac {y_i}{y_j}h\right)}{\vt(h)}\,,$$
which stands for the elliptic Chern class of the tangent bundle to the unipotent part of the Borel group. 
We define the reduced elliptic class:
$$\Ep^{\rm red}(\pa)=\frac{eu^{ell}_M~~\Ep(\pa)}{eu^{ell}_{\Fl}~~\mathcal B}\,.$$

\begin{example}\rm\label{Schred}
We illustrate how the elliptic classes of link patterns for $m=4$, $r=2$ specialize to the equivariant elliptic classes of $\Fl(2)=\PP^1$. This is a continuation of the basic example presented in \S\ref{basicexample}. Below we list the elliptic class of the patterns $\pa^{\rm min}_{4,2}$, the reduced class and its restriction to the fixed points of $\PP^1$. The first restriction is obtained by  substituting
$$y_1:=x_1\,,\qquad y_2:= x_2\,.$$
This fixed point is the 0-dimensional Schubert cell. The second fixed point is the center of the 1-dimensional cell. Restriction to that point is obtained by the substitution
$$y_1:=x_2\,,\qquad y_2:= x_1\,.$$
The resulting expressions are as follows:
\medskip

$
\begin{array}{l}
 \Ep(\pa^{\rm min}_{4,2})=\delta (\frac{x_1}{y_2},h)
   \delta (\frac{x_1}{y_1},\mu
   _1) \delta
   (\frac{x_2}{y_2},\mu _2)
   \\[0.3cm]
\Ep^{\rm red}(\pa^{\rm min}_{4,2})= \frac{\vt
   (\frac{x_2}{y_1}) \vt
   (\frac{h x_1}{y_2})
   \vt (\frac{\mu _1
   x_1}{y_1}) \vt
   (\frac{\mu _2
   x_2}{y_2})}{\vt (\mu
   _1) \vt (\mu
   _2) \vt
   (\frac{y_2}{y_1}) \vt
   (\frac{h y_1}{y_2})} \\[0.3cm]
\text{Restrictions to the fixed points}=(1,0) \\
\end{array}
$
\medskip

\noindent For $s_1\pa^{\rm min}_{4,2}$ we have:
\medskip

$\begin{array}{l}
\Ep(s_1\pa^{\rm min}_{4,2})= \delta (\frac{x_1}{x_2},h)
   \delta
   (\frac{x_2}{y_2},h)
   \delta (\frac{x_2}{y_1},\mu
   _1) \delta
   (\frac{x_1}{y_2},\mu
   _2)+\delta
   (\frac{x_2}{x_1},\frac{\mu
   _1}{\mu _2}) \delta
   (\frac{x_1}{y_2},h)
   \delta (\frac{x_1}{y_1},\mu
   _1) \delta
   (\frac{x_2}{y_2},\mu _2)
   \\[0.3cm]
\Ep^{\rm red}(s_1\pa^{\rm min}_{4,2})= \frac{\vt (\frac{h
   x_1}{x_2}) \vt
   (\frac{x_1}{y_1}) \vt
   (\frac{h x_2}{y_2})
   \vt (\frac{\mu _1
   x_2}{y_1}) \vt
   (\frac{\mu _2
   x_1}{y_2})}{\vt (h) \vt
   (\mu _1) \vt
   (\mu _2) \vt
   (\frac{x_1}{x_2}) \vt
   (\frac{y_2}{y_1}) \vt
   (\frac{h
   y_1}{y_2})}+\frac{\vt
   (\frac{\mu _1 x_2}{\mu _2
   x_1}) \vt
   (\frac{x_2}{y_1}) \vt
   (\frac{h x_1}{y_2})
   \vt (\frac{\mu _1
   x_1}{y_1}) \vt
   (\frac{\mu _2
   x_2}{y_2})}{\vt (\mu
   _1) \vt (\frac{\mu
   _1}{\mu _2}) \vt (\mu
   _2) \vt
   (\frac{x_2}{x_1}) \vt
   (\frac{y_2}{y_1}) \vt
   (\frac{h y_1}{y_2})} \\[0.3cm]
\text{Restrictions to the fixed points}=
\Big(\frac{\vt (\frac{\mu
   _1 x_2}{\mu _2 x_1})}{\vt
   (\frac{\mu _1}{\mu _2})
   \vt
   (\frac{x_2}{x_1})},\frac
   {\vt (\frac{h
   x_1}{x_2})}{\vt (h) \vt
   (\frac{x_1}{x_2})}\Big)=\big(\delta(\frac{x_2}{x_1},\frac{\mu_1}{\mu_2}),\delta(\frac{x_1}{x_2},h)\big) \\
\end{array}
$

\end{example}

\begin{theorem}The reduced class 
$\Ep^{\rm red}(\pa)$ represents the  elliptic characteristic class of the Schubert variety $X_{w_\pa}$ defined in \cite{RiWe1} after substitution $\mu_i:=\mu_i^{-1}$ and $x_{i+n}:=y_i$ for $1\leq i\leq n$.
\end{theorem}
\begin{proof} We check the R-matrix recursion \eqref{Rrecursion} for 
$\Ep^{\rm red}(\pa)$. 
By \eqref{eppadef}  we have
$$\Ep(s_i\pa)=\Cc^\ulu_i(\Ep(\pa))$$
whenever $i<2n-1$  and $s_i$ enlarges the dimensions of the corresponding $\B$-orbit.  
The function 
$eu^{ell}_M$
is symmetric with respect to permutation of $x_i$ variables and 
$eu^{ell}_{\Fl} \cdot \mathcal B$ does not depend on $x_i$ variables at all. 
Hence
$\Ep^{\rm red}(\pa)$ satisfies the recursion
$$\Ep^{\rm red}(s_i\pa)=\Cc^\ulu_i(\Ep^{\rm red}(\pa))$$
whenever $i<n$  and $\ell(s_iw_\pa)=\ell(w_\pa)+1$. 
We calculate the value of the admissible coefficient $\nu$ for the operation
$\Cc^\ulu_i=\Cc^\nu_i$. It is equal to the quotient of the labels of arrows pointing to $i$-th and $(i+1)$-th node, that is 
$$\frac{w^{-1}_\pa(\mu_i)}{w^{-1}_\pa(\mu_{i+1})}\,.$$
We  exchange the variables $\mu_i$ to $\mu_i^{-1}$ and obtain exactly  the R-matrix recursion \eqref{Rrecursion}. 

\medskip

It remains to examine  the restriction of  
$\Ep^{\rm red}(\pa^{\min}_{2n,n})$ to the fixed points of the flag variety. The torus fixed points in the flag variety are identified with the permutations $\sigma\in\Spe_n$. The restriction to the fixed point $\sigma$ is a function depending only on $x_i$ and $\mu_i$. It is obtained by the substitution $y_i:=x_{\sigma(i)}$ for $1\leq i\leq n$. 
We have to show that 
$$\Ep^{\rm red}(\pa^{\min}_{2n,n})_{|\{y_i:=x_{\sigma(i)}\}}=\begin{cases}1&\text{ if }\sigma=\id\\ 0&\text{ if }\sigma\neq\id\,.\end{cases}$$
We expand
\begin{align*}\Ep^{\rm red}(\pa^{\min}_{2n,n})&
=\frac
{\prod_{i=1}^{n}\prod_{j=1}^{n}\vt\big(\tfrac {x_i}{y_j}\big)\cdot \prod_{i=1}^n\delta\big(\tfrac{x_i}{y_i},\mu_i)\cdot \prod_{i<j}\delta\big(\tfrac{x_i}{y_j},h)}
{\prod_{i<j}\tfrac{\vt\left(\tfrac {y_i}{y_j}h\right)}{\vt(h)}\cdot \prod_{i>j}\vt\big(\tfrac {y_i}{y_j}\big)}\\
&=\frac
{\prod_{i=1}^{n}\prod_{j=1}^{n}\vt\big(\tfrac {x_i}{y_j}\big)\cdot
\prod_{i=1}^n\tfrac{\vt\big(\tfrac{x_i}{y_i}\mu_i\big)}{\vt\big(\tfrac{x_i}{y_i}\big)\vt(\mu_i)}
\cdot \prod_{i<j}\tfrac{\vt\big(\tfrac{x_i}{y_j}h\big)}{\vt\big(\tfrac{x_i}{y_j}\big)\vt(h)}}
{\prod_{i<j}\tfrac{\vt\left(\tfrac {y_i}{y_j}h\right)}{\vt(h)}
\cdot
\prod_{i>j}\vt\big(\tfrac {y_i}{y_j}\big)}
\\
&=\prod_{i>j}\frac{\vt\big(\tfrac {x_i}{y_j}\big)}{\vt\big(\tfrac {y_i}{y_j}\big)}
\cdot
\prod_{i=1}^n\frac{\vt\big(\tfrac{x_i}{y_i}\mu_i\big)}{\vt\big(\mu_i\big)}
\cdot \prod_{i<j}\frac{\vt\big(\tfrac{x_i}{y_j}h\big)}{\vt\left(\tfrac {y_i}{y_j}h\right)}
\,.
\end{align*}
The expression above is equal to 1 after substitution  $y_i:=x_i$. For other substitutions the first factor specializes to zero.
\medskip
\end{proof}

\begin{remark}\rm The Bott-Samelson recursion \eqref{BSrecursion}  follows directly from 
$$\Cc^\ulu_{n+i}(\Ep(\pa))=\Ep(s_{n+i}\pa),\qquad \text{if }\ell(w_\pa s_i)>\ell(w_\pa)\,.$$
\end{remark}

\section{Elliptic weight function}
The elliptic weight function of  \cite{RTV} in the form given in \cite[\S 6]{RiWe1} can be identified with the elliptic class of link patterns.
It is a function in $n$ parameters $z_i$ and $n-1$ parameters $\gamma_i$. Set
$$x_i=z_i\,,\text{ for } 1\leq i\leq n\text{ and }x_{j+n}=\gamma_j\,,\text{ for } 1\leq j\leq n-1\,.$$
Consider link patterns with $m=2n-1$ nodes and the rank $r=n-1$, such that the arrows have sources  in the last $n-1$ nodes. Such link patterns parameterize the orbits which are contained in the upper-right $n\times (n-1)$-rectangle. Each orbit coincides with an orbit with respect to $\B_n\times \B_{n-1}$ acting on $\Hom(\C^{n-1},\C^n)$. Since the considered link patterns have $n-1$ arrows, the orbits are of maximal rank, i.e. are contained in the set of injective maps 
$\Hom^{\rm inj}(\C^{n-1},\C^n)$. The quotient $\Hom^{\rm inj}(\C^{n-1},\C^n)/\B_{n-1}$ is just the flag variety $\Fl_n$ and the $\B_n\times \B_{n-1}$-orbits are mapped by the quotient map to the Schubert cells. As in the previous section we  multiply  the elliptic classes of link patterns by the elliptic class of the matrix block
$$eu^{ell}_{M'}=\prod_{i=1}^n\prod_{j=1}^{n-1}\vt\big(\tfrac {z_i}{\gamma_j}\big)$$
and divide by the Chern class of the unipotent part of $\B_{n-1}$
$$\mathcal B'=\prod_{n>i>j\geq 1}\tfrac{\vt\left(\tfrac {\gamma_i}{\gamma_j}h\right)}{\vt(h)}\,.$$
The resulting quotients 
\begin{equation}\label{wagowe}\frac{\Ep(\pa)\cdot eu^{ell}_{M'}}{\mathcal B'}\end{equation}
satisfies R-matrix recursion \eqref{Rrecursion} and for $\pa^{\min}_{2n-1,n-1}$ the restrictions to the fixed points of $\Fl_n$ are equal to 0 for $\sigma\neq \id$
and the elliptic Euler class for $\sigma=\id$. It is an exercise to check that the functions agree before restricting, provided that we introduce a substitution as below: 
\begin{corollary} The considered quotient \eqref{wagowe}
is equal to the elliptic weight function of \cite[\S6]{RiWe1} provided that we substitute \begin{equation}\label{RTVsubs}\mu_i:=\frac{h\mu_n}{\mu_i}\,, \text{ for }1\leq i<n\,.\end{equation}
\end{corollary}
\begin{example}\rm (Compare \cite[Example 6.2]{RiWe1}.) Let $n=3$, $m=5$, $r=2$. Then 
$$\Ep(\pa^{\min}_{5,2})=\delta\big(\tfrac {z_1}{\gamma_1},\mu_1\big)\delta\big(\tfrac {z_2}{\gamma_2},\mu_2\big)\delta\big(\tfrac {z_1}{\gamma_2},h\big)\,,$$
$$\frac{\Ep(\pa^{\min}_{5,2})\cdot eu^{ell}_{M'}}{\mathcal B'}
=\vt\big(\tfrac{z_2}{\gamma_1}\big)\,
\vt\big(\tfrac{z_3}{\gamma_1}\big)\,
\vt\big(\tfrac{z_3}{\gamma_2}\big)\,
\frac{\vt\big(\tfrac {z_1}{\gamma_2}h\big)}{\vt\big(\tfrac {\gamma_1}{\gamma_2}h\big)}\;
\frac{\vt\big(\tfrac {z_1}{\gamma_1}\mu_1\big)}
{\vt\big(\mu_1\big)}\;
\frac{\vt\big(\tfrac {z_2}{\gamma_2}\mu_2\big)}
{\vt\big(\mu_2\big)}\,.
$$
After the substitution \eqref{RTVsubs} we obtain $\hat{\bf  w}_{123}$ of {\it loc.cit.}
\end{example}
\section{Geometric meaning of $\Ep(\pa)$}
In the whole paper we have avoided to use directly the construction of elliptic characteristic classes as defined by Borisov and Libgober \cite{BoLi}, except for the computation for $\Ep(X^{\min}_{m,r})$. The starting case was trivial from the geometric point of view, and further we did not have to compute the boundary divisor discrepancies of the resolution maps  \eqref{BePerez}. We applied the principle {} that the elliptic class is pure. 
We expect that elliptic classes  represent sections of line bundles over a product of elliptic curves. This expectation is supported by explicit constructions of elliptic classes via stable envelopes; see \cite{RTV, RSZV, Smirnov, RiWe1}. The fact that the elliptic classes of Schubert varieties (after restrictions to the fixed points)  are pure is checked in \cite[\S8]{RiWe1}, however the name `pure'  was not introduced.
An axiomatic definition of the stable envelope  was given in \cite{AgOk}, where such classes are defined as homomorphisms of line bundles. It is widely conjectured that elliptic stable envelopes exist in great generality. Further examples, { for} the bow varieties, were constructed by Botta and Rimányi \cite{BottaRi}. Taken together, these results suggest that elliptic characteristic classes should be expressible as pure combinations of theta functions.

\medskip

Although our setting falls outside the formalism of stable envelopes, we  control of the coefficients $\nu_i$, so as to obtain a pure combination of $\delta$-functions. We show that this can indeed be achieved, and that the resulting class is independent of the choice of reduced word representing a given permutation.

We can easily read the multiplicities of the divisors. 
Suppose $\pa=s_{i_1}s_{i_2}\dots s_{i_\ell}\pa^{\min}_{m,r}$ is a reduced presentation of a link pattern.
Let $$\pi:Z=Z\strut_{m,r}^{\underline w}=P_{i_1}\times_\B  P_{i_2}\times_\B \dots \times_\B  P_{i_\ell}\times_\B  X^{\id}_{m,r}\to X^w_{m,r}$$
be the associated resolution and let $$Z'=P_{i_2}\times_\B  P_{i_3}\times_\B \dots \times_\B  P_{i_\ell}\times_\B  X^{\id}_{m,r}\,.$$
We have a fibration $Z=P_{i_1}\times_\B  Z'\to \PP^1=P_{i_1}/\B$ with the fiber $Z'$. The fiber over $\id \,\B$ is a component of the boundary divisor $\partial Z$, other components are obtained from the components of $\partial Z'$ by application of the associated bundle construction $P_{i_1}\times_\B  -$.
We define  the coefficients of the boundary of the resolution inductively. We start with the multiplicities $\lambda_i$ attached to the components of $\bigcup_{i=1}^mD_i=\partial X^{\min}_{m,r}$. Suppose the multiplicities $\alpha_i'$ of $\partial_i Z'$ are defined and we do not change them applying the associated bundle construction, only shifting the indices by 1.
It remains to define the coefficient $\alpha_1$ of the component $\B\times_\B  Z'\subset \partial Z$. Its value is determined by the admissible parameter $\nu$ of
$\Cc^\ulu_{i_1}=\Cc^\nu_{i_1}$.
We illustrate this with an example.

\begin{example}\rm Let 
$$\pa=s_1s_2\pa^{\min}_{3,1}=\xymatrix@-1.5pc{	
\tfrac{h^3}\alpha\ar@{..}[r]\ar@{->}@/^1.5pc/[r]&
\alpha h\ar@{..}[r]&
h^2}\,.$$
The link pattern $\pa$ represents the orbit of the matrix $\begin{pmatrix} 0&0&0\\1&0&0\\0&0&0\end{pmatrix}$.
The resolution (in one of the maps) has the form
$$(x,y,z)\mapsto 
\begin{pmatrix} 1&0&0\\z&1&0\\0&0&1\end{pmatrix}
\begin{pmatrix} 1&0&0\\0&1&0\\0&y&1\end{pmatrix}\cdot
\begin{pmatrix} 0&0&x\\0&0&0\\0&0&0\end{pmatrix}=\begin{pmatrix}  x y z & -x y & x \\
 x y z^2 & -x y z & x z \\
 0 & 0 & 0\end{pmatrix}\,.$$
Here the boundary divisors are the following:
$$\widetilde D_1=\{x=0\}\,,$$
$$\partial_2Z=\{y=0\}\,,$$
$$\partial_1 Z=\{z=0\}\,.$$
We set the multiplicity of $\widetilde D_1$ to be $\lambda_1=\lambda$ and 
$\mu=h^{1-\lambda}$, which in the additive notation means $\mu=(1-\lambda)h$.
The multiplicities of $\partial_i Z$ are dictated by the diagram
$$\begin{matrix}\\[30pt]s_2\\[35pt]s_1\\[20pt]\end{matrix}\hskip20pt\begin{tikzcd}[nodes={inner sep=0pt}] 
  &\arrow[-,dashed]{dd}\\ \\
  \boxed{\mu h}\arrow[-]{d}
    &\boxed{ h^2}\arrow[-,dashed]{dr}{\hskip-4pt\frac{\mu}{h}}
    & \boxed{\tfrac {h^3}{\mu}} \arrow[-]{dl}\arrow[bend right=60,swap]{ll}{\hskip-10pt\mu}
        \\
\boxed{   \mu h} \arrow[-]{dr} {\hskip-7pt\frac{\mu^2}{h^2}}
    & \boxed{\tfrac {h^3}\mu}\arrow[-]{dl} 
    & \boxed{h^2}\arrow[-,dashed]{d}
    \\
\boxed{      \tfrac{h^3}\mu}& \boxed{\mu h }&\boxed{ h^2}  \\
\end{tikzcd}\hskip30pt
\begin{matrix}\\[20pt]\frac{h^2}{h^3/\mu}=\frac{\mu}h\\[23pt]\frac{\mu h}{h^3/\mu}=\frac{\mu^2}{h^2}\\[20pt]\end{matrix}
$$
The label over the upper crossing  
is 
$$\frac \mu h =\frac{h^{1-\lambda}}h=h^{-\lambda}\,.$$
This tells us that the multliplicity $\alpha_2$ has to satisfy
$$h^{1-\alpha_2}=h^{-\lambda}\,.$$
The bottom label is $$\frac{\mu^2}{h^2}=\frac{h^{2(1-\lambda)}}{h^2}=h^{-2\lambda}$$
implies 
$$1-\alpha_1=-2\lambda\,.$$
Finally we find the coefficients of the boundary divisor in the Bott-Samelson resolution:
$$\alpha_1=2\lambda+1\,,\qquad \alpha_2=\lambda+1\,.$$

\end{example}
We will not write the general formula for multiplicities. It is just a simple application of the inductive procedure of determining the admissible value of the parameter in Demazure operation.



\end{document}